\documentclass[a4paper,12pt]{article}
\usepackage{graphicx} 

\usepackage{amsmath,amssymb,amsfonts,amsthm}
\usepackage{multirow,bigdelim}
\usepackage{amscd}
\usepackage{verbatim}
\usepackage{latexsym}
\usepackage{array}
\usepackage{enumerate}
\usepackage[all]{xy}
\usepackage{graphicx}
\usepackage{ulem}
\usepackage{ascmac}
\usepackage{mathtools}
\usepackage{layout}
\usepackage{bbm,enumitem,empheq,nicematrix,relsize,calligra}
\usepackage[initials,alphabetic]{amsrefs}
   \BibSpec{arXiv}{
   +{}{\PrintAuthors}{author}
   +{,}{ \textit}{title}
   +{,}{ }{date}
   +{,}{ arXiv:}{eprint}}
\usepackage{mleftright}
   \mleftright
\usepackage{tikz}
\usetikzlibrary{cd,decorations,decorations.pathreplacing,arrows,calc,intersections,hobby,decorations.pathmorphing} 
\usepackage{adjustbox}
\usepackage{extpfeil}
\usepackage[T5]{fontenc}

\numberwithin{equation}{section}

\newcommand{\BDC}{{\mathbf{D}}^{\mathrm{b}}}

\newcommand{\Mod}{\mathrm{Mod}}

\newcommand{\rsect}{{\mathrm{R}}\Gamma}
\newcommand{\CC}{\mathbb{C}}

\newcommand{\RR}{\mathbb{R}}

\newcommand{\ZZ}{\mathbb{Z}}

\newcommand{\codim}{{\rm codim}}

\newcommand{\e}{\varepsilon}

\newcommand{\Perv}{{\rm Perv}}
\newcommand{\Sol}{{\rm Sol}}

\newcommand{\simto}{\overset{\sim}{\longrightarrow}}

\newcommand{\CF}{{\rm CF}}

\newcommand{\op}{\mbox{\scriptsize op}}
\newcommand{\SD}{\mathcal{D}}

\newcommand{\SO}{\mathcal{O}}

\newcommand{\SM}{\mathcal{M}}

\newcommand{\SF}{\mathcal{F}}

\newcommand{\Modhol}{\mathrm{Mod}_{\mathrm{hol}}}
\newcommand{\Modrh}{\mathrm{Mod}_{\mbox{\scriptsize rh}}}

\newcommand{\rhom}{{\rm R}{\mathcal{H}}om}

\newcommand{\ch}{\rm char}

\renewcommand{\ch}{{\mathrm{char}}}

\DeclareMathOperator{\supp}{supp}

\DeclareMathOperator{\mult}{mult}
\DeclareMathOperator{\CCyc}{CC} 
\DeclareMathOperator{\msupp}{SS} 

\DeclareRobustCommand{\longtwoheadrightarrow}{\relbar\joinrel\twoheadrightarrow} 
\DeclareRobustCommand{\longtwoheadleftarrow}{\twoheadleftarrow\joinrel\relbar}

\newcommand{\longhookrightarrow}{\lhook\joinrel\longrightarrow}

\DeclarePairedDelimiter{\abs}{\lvert}{\rvert} 
 
\DeclarePairedDelimiterX{\Set}[2]{\lbrace}{\rbrace}{#1\ \delimsize\vert\ #2}

\newtheorem{theorem}{Theorem}[section]

\newtheorem{corollary}[theorem]{Corollary}
\newtheorem{lemma}[theorem]{Lemma}
\newtheorem{proposition}[theorem]{Proposition}

\theoremstyle{definition}
\newtheorem{definition}[theorem]{Definition}
\theoremstyle{remark}
\newtheorem{remark}[theorem]{\sc Remark}
\theoremstyle{remark}
\newtheorem{example}[theorem]{\sc Example}

\title{Euler obstructions and Verdier specializations
}

\author{Ren FERNANDES 
\footnote{Mathematical Institute, Tohoku University,
Aramaki Aza-Aoba 6-3, Aobaku, Sendai, 980-8578, Japan.
E-mail: fernandes.ren.p7@dc.tohoku.ac.jp}
and Kiyoshi TAKEUCHI 
\footnote{Mathematical Institute, Tohoku University,
Aramaki Aza-Aoba 6-3, Aobaku, Sendai, 980-8578, Japan.
E-mail: takemicro@nifty.com} }

\begin{document}
\maketitle
\begin{abstract}
To establish a basis of the sheaf theoretical study of the
Milnor fibers and monodromies of complete
intersection varieties, we clarify the structures
of the Verdier specialization sheaves associated
to them. Assuming the Thom condition in a direction,
some basic results on them will be obtained.
Then we apply them to obtain formulas for the
Euler obstructions of complete intersection varieties
having non-isolated singular points.
\end{abstract}

\section{Introduction}
The aim of this paper is twofold.
First, we study of the structures of the Verdier specialization sheaves 
associated to complete intersection varieties, for which almost nothing 
has been known until now. 
In particular, assuming the Thom condition for them, 
a basic relation between their Milnor fibers and the stalks of the 
Verdier specializations will be obtained.
Moreover, based on the result of \cite{FKT26} we will describe the 
characteristic cycles of the Verdier specializations in terms of some Milnor fibers. 
We thus establish a solid basis for the sheaf theoretical study of the 
Milnor fibers and monodromies of complete intersection varieties. 
Our second aim is to apply these results to obtain formulas for the 
Euler obstructions of C.I. varieties with possibly non-isolated singular points. 
Recall that Euler obstructions of complex varieties and analytic sets 
were introduced by Kashiwara \cite{Kas73}, \cite{Kas83a} and 
MacPherson \cite{Mac74} independently, more than fifty years ago. 
Euler obstructions are of exceptional importance in various fields of 
mathematics, such as $\SD$-module theory, topology, singularity theory, 
algebraic geometry and combinatorics. 
However, despite a lot of effort by many mathematicians, even until now 
there has been no universal way to describe them in general. 
An important partial result in \cite{LT81} 
reduces the calculation of the value 
of the Euler obstruction of a variety 
at a point to those of its polar multiplicities there. 
For varieties with non-isolated singular points, 
formulas of their Euler obstructions are known only for very special classes of 
varieties such as toric and determinantal varieties (see e.g. 
\cite{MT11}, \cite{GGR19} and \cite{LR22}). 
For excellent surveys on Euler obstructions, 
see \cite{BGB22} and \cite{BSS09} etc. 
Note that Euler obstructions are important also in the study of 
projective duality (see e.g. \cite{Ern94}, \cite{MT08}, \cite{MT11}). 
We hope that our formulas for C.I. varieties will shed some light 
in the further study of Euler obstructions. 

\medskip \indent 
In order to explain our results more precisely, let us 
prepare some notations and definitions. 
Here we present our results only in the algebraic situation. 
As is clear from our proof, similar results hold true also 
for complex analytic C.I. varieties. Moreover, the 
case of complex hypersurfaces being a very special one 
of C.I. varieties in which some of our assumptions are always 
satisfied, the results for complex hypersurfaces are 
obtained much more easily (see Section \ref{sec-hypersurface}). 
Let $Y=\{f_1=\cdots =f_d=0\} \subset X=\CC^n \quad (f_i(x)\in \CC[x_1,\ldots x_n])$
be a complete intersection subvariety of $X=\CC^n$ such that $0\in Y$ and 
\begin{equation}
i_f\colon X\longhookrightarrow X\times \CC_s^d \quad (x\longmapsto (x,f_1(x),\ldots , f_d(x)))
\end{equation}
the graph embedding of $f=(f_1,\ldots , f_d) \colon X\longrightarrow \CC_s^d$. 
Note that for the normal bundle $T_{X\times \{0\}}(X\times \CC^d)$ of 
$X\times \{0\} \subset X\times \CC^d$ in $X\times \CC^d$ 
we have a natural isomorphism 
\begin{equation}
T_{X\times \{0\}}(X\times \CC^d) \simeq X\times \CC^d
\end{equation}
(see e.g. \cite[Appendix B.6]{Ful98}). 
Then to a constructible sheaf $F\in \BDC_c(X)$ on $X=\CC^n$ the 
Verdier specialization functor introduced in \cite{Ver83} associates a one 
$\mathrm{Sp}_{X\times \{0\}|X\times \CC^d}({i_f}_\ast F) \in \BDC_c(X\times \CC^d)$ 
on $T_{X\times \{0\}}(X\times \CC^d) \simeq X\times \CC^d$. 
Since this operation preserves the perversity, 
for $F=\CC_X[n] \in \Perv(\CC_X)$ we obtain a perverse sheaf 
\begin{equation}
\mathcal{G} \coloneq \mathrm{Sp}_{X\times \{0\}| X\times \CC^d}({i_f}_\ast 
\CC_X[n]) \quad \in \Perv(\CC_{X\times \CC^d})
\end{equation}
on $T_{X\times \{0\}}(X\times \CC^d) \simeq X\times \CC^d$. 
Just like Deligne's nearby cycle sheaves introduced in \cite{Del73} 
are related to the Milnor fibers and 
monodromies of holomorphic functions, 
this Verdier specialization sheaf $\mathcal{G} \in \Perv(\CC_{X \times \CC^d})$ 
should be related to those of the morphism $f=(f_1,\ldots , f_d)\colon X\longrightarrow \CC^d$. 
Indeed, without establishing such a relation, several authors already used 
the Verdier specialization to study the higher codimensional analogue of 
Milnor's theory in \cite{Mil68} (see e.g. \cite{Bud13} and \cite{VPV10}). 
To have such a relation, in fact we need Thom's condition 
(see Definition \ref{def-thom}, \cite{Sea06} and \cite{Sea19}) as we will see below. 
Let $\Delta_f \subset \CC_s^d$ be the discriminant set of $f$ and 
$C_{\{0\}}(\Delta_f) \subset T_{\{0\}}\CC^d \simeq \CC^d$ its cone along the origin $\{0\}\subset \CC^d$. 
If $f\colon X\longrightarrow \CC^d$ satisfies the Thom condition in a direction 
$a \in \CC^d \setminus C_{\{0\}}(\Delta_f)$, 
we can show that its Milnor fiber at the origin $0\in Y\subset X=\CC^n$ 
(in the direction $a$) 
exists and we denote it by $M_{Y,0}\subset X\setminus Y$. 
Then we obtain the following result. 

\begin{proposition}\label{prop-milnor-fiber}
Assume that $f=(f_1,\ldots ,f_d)\colon X\longrightarrow \CC_s^d$ satisfies 
the Thom condition in a direction $a\in \CC^d \setminus C_{\{0\}}(\Delta_f)$. 
Then there exist isomorphisms 
\begin{align}
H^j\mathcal{G}_{(0,a)} &\simeq H^j \mathrm{Sp}_{X\times \{0\} |X\times \CC^d}({i_f}_\ast \CC_X[n])_{(0,a)}\\
&\simeq H^{j+n}(M_{Y,0};\CC) \quad (j\in \ZZ).
\end{align}
\end{proposition}

Furthermore, we find the following nice structure of the 
perverse sheaf $\mathcal{G} \in \Perv(\CC_{X \times \CC^d})$. 

\begin{proposition}\label{flat-Ver} 
In the situation of Proposition \ref{prop-milnor-fiber}, 
let $U\subset \CC_s^d$ be a sufficiently small neighborhood of the point 
$a\in \CC^d \setminus C_{\{0\}}(\Delta_f)$ in $\CC^d \setminus C_{\{0\}}(\Delta_f)$
and $\alpha_U \colon X\times U \longrightarrow X$ the projection. 
Then there exists a perverse sheaf $G\in \Perv(\CC_X)$ on $X=\CC^n$ for which 
we have an isomorphism 
\begin{equation}
\mathcal{G}|_{X\times U} \simeq \alpha_U^{-1}G[d].
\end{equation}
\end{proposition}

Then, together with the results in \cite{FKT26} we obtain the following formula 
for the Euler obstruction $\mathrm{Eu}_Y$ of the C.I. subvariety 
$Y=\{f_1=\cdots =f_d=0 \} \subset X=\CC^n$. In what follows in 
this section, we always assume that 
$f=(f_1,\ldots ,f_d)\colon X\longrightarrow \CC_s^d$ satisfies 
the Thom condition in a direction $a\in \CC^d \setminus C_{\{0\}}(\Delta_f)$. 
Let $H \ (\simeq \CC^{n-1}) \subset X=\CC^n$ be a 
generic hyperplane passing through the origin 
$0 \in Y \subset X= \CC^n$. Then, if for 
$g\coloneq f|_H\colon H\longrightarrow \CC^d$ the 
condition $a\in \CC^d \setminus C_{\{0\}}(\Delta_{g})$ is satisfied, 
we can show that the Milnor fiber $M_{Y\cap H,0} \subset H\setminus{Y\cap H}$ of 
$g$ at the origin $0\in Y\cap H =g^{-1}(0) \subset H$ (in the direction 
$a$) exists (see Theorem \ref{thm-mult} (i)). 
Let $Y=\bigsqcup_{\alpha \in A} Y_\alpha$ be a Whitney stratification of $Y$ 
satisfying some conditions (see Section \ref{Eu-CI}) 
such that $\{0\} =Y_\alpha$ for some 
$\alpha \in A$ that we denote by $\alpha_0$. 
For $\alpha \in A$ such that $\alpha \neq \alpha_0, 0\in \overline{Y_\alpha}$ 
and $\dim Y_\alpha <\dim Y$ let $W_\alpha \ (\simeq \CC^{n-\dim Y_\alpha}) 
\subset X=\CC^n$ be an affine subspace which intersects the stratum $Y_\alpha$ 
at a point $p_\alpha \in Y_\alpha$ transversally. 
We call it a normal slice of $Y_\alpha$ at $p_\alpha \in W_\alpha \cap Y_\alpha$. 
Let $U_\alpha \subset W_\alpha$ be a sufficiently small neighborhood of 
$p_\alpha$ in $W_\alpha \simeq \CC^{n-\dim Y_\alpha}$ and 
$M_{Y\cap U_\alpha, p_\alpha} \subset U_\alpha \setminus{Y\cap U_\alpha}$ 
the Milnor fiber of $f_\alpha \coloneq f|_{U_{\alpha}} 
\colon U_\alpha \longrightarrow \CC^d$ at 
the point $p_\alpha \in Y\cap U_\alpha$ (in the direction $a$). 
Let $H_\alpha \ (\simeq \CC^{n-\dim Y_\alpha -1}) 
\subset W_\alpha$ be a generic hyperplane 
in $W_\alpha \simeq \CC^{n-\dim Y_\alpha}$ passing through the point 
$p_\alpha \in U_\alpha \subset W_\alpha$. Then, if 
for $g_\alpha \coloneq f|_{H_\alpha \cap U_\alpha} \colon 
H_\alpha \cap U_\alpha \longrightarrow \CC^d$ the condition 
$a\in \CC^d \setminus C_{\{0\}}(\Delta_{g_\alpha})$ is 
satisfied, we can define also the Milnor fiber 
$M_{Y\cap H_\alpha \cap U_\alpha, p_\alpha} \subset (H_\alpha \cap U_\alpha) 
\setminus (Y\cap H_\alpha \cap U_\alpha)$ of $g_{\alpha}$ at the point 
$p_\alpha \in Y\cap H_\alpha \cap U_\alpha$ (in the direction $a$).  

\begin{theorem}\label{thm-euler-formula}
In the situation of Proposition \ref{prop-milnor-fiber}, assume also 
that $a\in \CC^d \setminus C_{\{0\}}(\Delta_{g})$ and 
for any $\alpha \in A$ such that $\alpha \neq \alpha_0, 0\in \overline{Y_\alpha}$ 
and $\dim Y_\alpha <\dim Y$ we have $a\in \CC^d \setminus C_{\{0\}}(\Delta_{g_\alpha})$. 
Then for the Euler obstruction $\mathrm{Eu}_Y$ of $Y$ we have 
\begin{align}
\mathrm{Eu}_Y(0)=&\chi(M_{Y\cap H,0})\\
&- \sum_{\substack{\alpha \neq \alpha_0, 0\in \overline{Y_\alpha} \\\dim 
Y_\alpha <\dim Y}} \left\{ \chi(M_{Y\cap U_\alpha, p_\alpha}) -\chi(M_{Y\cap H_\alpha 
\cap U_\alpha, p_\alpha}) \right\} \cdot \mathrm{Eu}_{\overline{Y_\alpha}}(0).
\end{align}
\end{theorem}
For the proof of Theorem \ref{thm-euler-formula}, in addition to Proposition 
\ref{flat-Ver} we are indebted to some ideas and results in \cite{CMSS16}, 
\cite{FKT26}, \cite{Gin86}, \cite{KS90} and \cite{Le73}. 
Note that if the C.I. variety $Y=\{f_1=\cdots =f_d=0\} \subset X=\CC^n$ has an 
isolated singular point at the origin $0\in Y$ then the Thom condition is 
satisfied for any direction $a\in \CC^d \setminus C_{\{0\}}(\Delta_f)$ and 
we have a Whitney stratification $Y=(Y\setminus \{0\})\sqcup \{0\}$ of $Y$. 
Hence the following classical result due to Kashiwara \cite{Kas83a}, 
Dubson \cite{Dub78}, Kato \cite{Kat78} and Piene \cite{Pie88} 
is a very special case of Theorem \ref{thm-euler-formula}. 
\begin{corollary}
Assume that the C.I. subvariety $Y=\{f_1=\cdots =f_d=0\} 
\subset X=\CC^n$ has an isolated singular point at the origin $0\in Y$. 
Then we have 
\begin{equation}
\mathrm{Eu}_Y(0)=\chi(M_{Y\cap H,0}) =1+(-1)^{n-d-1}\mu_{Y\cap H, 0},
\end{equation}
where $\mu_{Y\cap H}=\dim_{\CC} H^{n-d-1}(M_{Y\cap H,0}; \CC )$ is the Milnor number of 
the C.I. subvariety $Y\cap H \subset H$ of $H\simeq \CC^{n-1}$ at the origin 
$0\in Y\cap H$ (c.f. Hamm \cite{Ham71}).
\end{corollary}
Theorem \ref{thm-euler-formula} allows us to reduce the calculation of $\mathrm{Eu}_Y(0)$ to 
those of $\mathrm{Eu}_{\overline{Y_\alpha}}(0)$ for $\alpha \in A$ 
such that $\dim Y_\alpha <\dim Y$. 
Note that for $\alpha \in A$ such that $\dim Y_\alpha =1$ the 
calculation of $\mathrm{Eu}_{\overline{Y_\alpha}}(0)$ is very easy 
(see e.g. \cite[Theorem 3.1]{BLS00}). 
Hence in the case where the dimension of the singular set 
$Y_{\mathrm{sing}}\subset Y$ of $Y$ 
is $\leq 1$ we are done. 
Even in the case where $\dim Y_{\mathrm{sing}} \geq 2$, 
if $Y_{\mathrm{sing}}$ is a C.I. subvariety of $X=\CC^n$ we can apply 
Theorem \ref{thm-euler-formula} once again to reduce the calculation of 
$\mathrm{Eu}_{Y_{\mathrm{sing}}}(0)$ 
to those of lower dimensional subvarieties of $Y_{\mathrm{sing}}$. 
It would be a challenging problem to extend 
Theorem \ref{thm-euler-formula} to non C.I. subvariety of $X=\CC^n$. 

\medskip \indent 
Finally, we shall discuss the validity of our assumption on the C.I. subvariety 
$Y=\{f_1=\cdots =f_d=0\}\subset X=\CC^n$ i.e. to what extent 
$f\colon X\longrightarrow \CC^d$ satisfies the Thom condition in a direction 
$a\in \CC^d \setminus C_{\{0\}}(\Delta_f)$. 
If $d=1$ and $Y=\{f=0\}\subset X=\CC^n$ is a complex hypersurface, this condition is 
always satisfied and the Milnor fiber of $f$ at the origin $0\in Y$ exists. 
Note that in this case we have $\Delta_f=\{0\} \subset \CC$. 
However, to our surprise, if $d>1$ our assumption does not always hold true. 
Indeed, in Example \ref{ex-coexample} we construct an example of a C.I. subvariety 
$Y=\{f_1=\cdots =f_d=0 \} \subset X=\CC^n$ for which $f\colon X\longrightarrow \CC^d$ 
does not satisfy the Thom condition 
in any direction $a\in \CC^d \setminus C_{\{0\}}(\Delta_f)$. 
For this purpose, we use a classical result of Hironaka in \cite{Hir76}.
See Section \ref{sec-verdier} for the details. 
In Proposition \ref{prop-af} and Example \ref{ex-str-trans}, 
nevertheless we also develop a method to construct various examples of 
$Y=\{f_1=\cdots =f_d=0 \} \subset X=\CC^n$ for which $f\colon X\longrightarrow \CC^d$ 
satisfies the Thom condition in a direction $a \in \CC^d \setminus C_{\{0\}}(\Delta_f)$. 
At this moment, it is likely that for a generic choice of $f_1,\ldots , f_d \in 
\CC[x_1,\ldots , x_n]$ our assumption on the C.I. subvariety 
$Y=\{f_1=\cdots =f_d=0 \} \subset X=\CC^n$ holds true.

\bigskip
\noindent{\bf Acknowledgement:}
The authors would like to express their
heartfelt gratitude to Professor Toshizumi Fukui 
for useful discussions with him 
during the preparation of this paper.

\section{Preliminary notions and results}
In this section, we recall the definitions of some basic notions 
which will be used in this paper and explain their fundamental properties. 
For this purpose, we essentially follow the terminologies in \cite{Dim04}, 
\cite{HTT08}, \cite{Kas03}, \cite{KS90} and \cite{Sch03}. 
For a topological space $X$ we denote by $\BDC(X)$ the derived category 
consisting of bounded complexes of sheaves of $\CC_X$-modules on $X$. 
In what follows, we assume that $X$ is a complex manifold of dimension $n$. 
We denote by $X_{\RR}$ its underlying real analytic manifold. 
As in \cite[page 344]{KS90} we shall say that a subset $S\subset X$ is 
$\CC$-analytic if $\overline{S}$ and $\partial S\coloneq \overline{S} \setminus S$ 
are analytic subsets of $X$. We can easily see that $S\subset X$ is $\CC$-analytic 
if and only if $S=V\setminus W$ for some analytic subsets $V, W\subset X$ of $X$. 
It follows that if $S_1,S_2\subset X$ are $\CC$-analytic their intersection 
$S_1\cap S_2$ is also $\CC$-analytic. 
\begin{definition}
Let $Y\subset X$ be an analytic subset of $X$. Then a locally finite partition 
$Y=\bigsqcup_{\alpha \in A} Y_\alpha$ of $Y$ into $\CC$-analytic complex submanifolds 
$Y_\alpha \subset X$ of $X$ is called a (complex analytic) stratification of $Y$ 
if for any $\alpha \in A$ there exists a subset $B\subset A$ such that 
$\overline{Y_\alpha} =\bigsqcup_{\beta \in B}Y_\beta$. 
For $\alpha \in A$ the complex submanifold $Y_\alpha \subset X$ is called a stratum 
of the stratification $Y=\bigsqcup_{\alpha \in A}Y_\alpha$.
\end{definition}
A $\CC$-analytic subset $S\subset X$ is complex constructible in the sense of the 
following definition. 
\begin{definition}
We say that a subset $S\subset X$ is (complex) constructible if there exists a 
(complex analytic) stratification $X=\bigsqcup_{\alpha \in A}X_\alpha$ of $X$ and a 
subset $B\subset A$ such that $S=\bigsqcup_{\beta \in B}X_\beta$.
\end{definition}
We denote the family of (complex) constructible subsets of $X$ by $\mathcal{C}_X$. 
Then it is easy to see that if $S_1, S_2 \in \mathcal{C}_X$ then 
$S_1\cap S_2, \ S_1\cup S_2 \in \mathcal{C}_X$. 
Constructible subsets of complex manifolds enjoy the following nice functorial properties for morphisms among them. 
\begin{lemma}\label{lem-const}
Let $f\colon X\longrightarrow Y$ be a morphism of complex manifolds. 
Then we have 
\begin{enumerate}
\item [(i)] If $S\in \mathcal{C}_Y$ then $f^{-1}(S)\in \mathcal{C}_X$. 
\item [(ii)] If $f$ is proper and $S\in \mathcal{C}_X$ then $f(S) \in \mathcal{C}_Y$.
\end{enumerate}
\end{lemma}
\begin{proof}
(i) The problem being local on $Y$, we may assume that $S$ is a finite union of 
$\CC$-analytic subsets $S_i \subset Y \ (i\in I)$. 
For any $i\in I$ we can easily show that $f^{-1}(S_i) \subset X$ is $\CC$-analytic 
and hence $f^{-1}(S_i) \in \mathcal{C}_X$. 
It follows that we have $f^{-1}(S)=\bigcup_{i\in I} f^{-1}(S_i) \in \mathcal{C}_X$. \\
(ii) The assertion immediately follows from the results in \cite[page 43]{GM88}.
\end{proof}
\begin{definition}
Let $Y\subset X$ be a complex analytic subset of $X$. 
A stratification $Y=\bigsqcup_{\alpha \in A}Y_\alpha$ of $Y$ is called a Whitney 
stratification if it satisfies the following conditions (a) and (b): 
\begin{enumerate}
\item [(a)]
Assume that a sequence $x_i\in Y_\alpha$ of points converges to a point $y\in Y_\beta \ 
(\alpha \neq \beta)$ and the limit $\mathbf{T}$ of the tangent spaces $T_{x_i}S_\alpha$ exists. 
Then we have $T_y Y_\beta \subset \mathbf{T}$. 
\item [(b)] 
Let $x_i \in Y_\alpha$ and $y_i \in Y_\beta$ be two sequences of points which converge to 
the same point $y\in Y_\beta \ (\alpha \neq \beta)$. Assume further that the limit $\ell$ 
(resp. $\mathbf{T}$) of the complex lines $\ell_i$ joining $x_i$ and $y_i$ (resp. of the 
tangent spaces $T_{x_i}Y_\alpha$) exists. 
Then we have $\ell \subset \mathbf{T}$. 
\end{enumerate}
\end{definition}
It is well known that any stratification of an analytic set can be refined to satisfy the 
Whitney conditions. In addition, we prepare a basic property of 
Whitney stratifications and prove it for the convenience of readers. 

\begin{lemma}\label{lem-hironaka}
Let $Y\subset X=\CC^n$ be a complex analytic C.I. subvariety of $X= \CC^n$ 
and $Y=\bigsqcup_{\alpha \in A} Y_\alpha$ its complex 
analytic stratification satisfying the 
Whitney conditions (a) and (b). Then any stratum $Y_\alpha$ in it is contained in either 
$Y_{\mathrm{reg}}\coloneq Y\setminus Y_{\mathrm{sing}}$ or $Y_{\mathrm{sing}}$.
\end{lemma}
\begin{proof}
Suppose that there exists a stratum $Y_\alpha$ such that 
$Y_\alpha \cap Y_{\mathrm{reg}}\neq \varnothing$ and 
$Y_\alpha \cap Y_{\mathrm{sing}} \neq \varnothing$. 
By \cite[Theorem 1.5.9]{Tro20}, the multiplicity of $Y$ is constant on 
any stratum of $Y$. Since $Y_\alpha \cap Y_{\mathrm{reg}} \neq \varnothing$, 
the multiplicity of $Y$ is equal to $1$ on the whole $Y_\alpha$. 
It follows from the multiplicity one criterion that 
$Y_\alpha \subset Y_{\mathrm{reg}}$ (see e.g. \cite[Theorem 6.8]{HIO88}). 
This is a contradiction to the condition $Y_\alpha \cap Y_{\mathrm{sing}} \neq \varnothing$. 
\end{proof}
As we see in the example below, 
Lemma \ref{lem-hironaka} does not always hold for stratifications satisfying 
only the Whitney (a) condition. 
\begin{example}
Let us consider a complex hypersurface 
$Y\coloneq \{x^2+y^2+xz^2=0\} \subset \CC_{x,y,z}^3$ and take its stratification
$Y=\bigsqcup_{i=0,1} Y_i$ as follows: 
\begin{equation}
\begin{cases}
Y_0\coloneq \{(0,0,z) \mid z\in \CC \} \\
Y_1\coloneq Y\setminus Y_0.
\end{cases}
\end{equation}
It is easy to see that this stratification satisfies the Whitney (a) condition. 
However, the Whitney (b) condition does not hold. 
Indeed, let us take two sequences of points $p_t=(-t^2,0,t) \in Y_1 \ (t\in \CC)$ and 
$q_t=(0,0,t)\in Y_0 \ (t\in \CC^*)$. 
Then the limit of the complex lines joining $p_t$ and $q_t$ (resp. 
the tangent spaces $T_{p_t} Y_1$) as $t \longrightarrow 0$ 
is $\mathrm{Span}_\CC \{ \partial_x\}$ (resp. 
$\mathrm{Span}_\CC \{ \partial_y, \partial_z\}$). 
On the other hand, since this hypersurface $Y$ has an isolated singular 
point at the origin $0\in \CC_{x,y,z}^3$ 
i.e. $Y_{\mathrm{sing}}=\{0\}$,  
we obtain $Y_0 \cap Y_{\mathrm{reg}}\neq \varnothing$ and $Y_0 \cap Y_{\mathrm{sing}} \neq \varnothing$.
\end{example}
\begin{definition}
Let $F$ be a sheaf of $\CC_X$-modules on $X$. 
Then we say that $F$ is (complex) constructible if there exists a (complex 
analytic) stratification 
$X=\bigsqcup_{\alpha \in A}X_\alpha$ of $X$ such that for any $\alpha \in A$ 
we have an isomprphism $F|_{X_\alpha}\simeq\CC_{X_\alpha}^{\oplus r_\alpha}$ 
locally on $X_\alpha$ for some $r_\alpha \geq 0$. 
\end{definition}
More generally, for an object $F\in \BDC(X)$ of the derived category 
$\BDC(X)$ we say that $F$ is (complex) constructible if its cohomology sheaf $H^jF$ 
is constructible for any $j\in \ZZ$. 
We denote by $\BDC_c(X) \subset \BDC(X)$ the full subcategory of $\BDC(X)$ consisting of 
constructible objects. 
Similarly, we can define $\ZZ$-valued 
constructible functions on $X$ as follows. 
\begin{definition}
Let $\varphi \colon X\longrightarrow \ZZ$ be a $\ZZ$-valued function on $X$. 
Then we say that $\varphi$ is (complex) constructible if there exists a (complex analytic) stratification 
$X=\bigsqcup_{\alpha \in A}X_\alpha$ of $X$ such that for any $\alpha \in A$ the restriction 
$\varphi|_{X_\alpha} \colon X_\alpha \longrightarrow \ZZ$ of $\varphi$ to the stratum 
$X_\alpha \subset X$ is constant. 
\end{definition}
We set $\CF_{\ZZ}(X)\coloneq \{\varphi \colon X\longrightarrow \ZZ \quad \text{constructible functions} \}$. 
Then it turns out that $\CF_{\ZZ}(X)$ is naturally an abelian group for the additions of functions. 
Constructible sheaves and functions are 
related as follows. 
Let $F\in \BDC_c(X)$ be a constructible sheaf (object) on $X$. 
We define a $\ZZ$-valued function $\chi(F) \colon X \longrightarrow \ZZ$ by 
\begin{equation}
\chi(F)(x)\coloneq \sum_{j\in \ZZ}(-1)^j\dim_{\CC}(H^jF)_x \quad (x\in X).
\end{equation}
Then we can easily check that $\chi(F)$ is a constructible function on $X$ 
i.e. $\chi(F) \in \CF_{\ZZ}(X)$. 
Among the constructible functions on $X$, 
the following ones defined for analytic subsets $Y\subset X$ of $X$ are of exceptional importance. 
Here we recall Kashiwara's definition in \cite{Kas73} and \cite{Kas83a}. 
Let $Y\subset X$ be an irreducible analytic subset of $X$ and $Y=\bigsqcup_{\alpha \in A}Y_\alpha$ 
its Whitney stratification consisting of connected strata $Y_\alpha$. 
Then we define a constructible function $\varphi \colon X\longrightarrow \ZZ$ such that 
$\varphi|_{X\setminus Y} \equiv 0$ and $\varphi|_{Y_\alpha}$ is constant for any $\alpha \in A$ as follows. 
First, for $\alpha \in A$ such that $\codim_Y Y_\alpha =0$ we set $\varphi|_{Y_\alpha} \equiv 1$. 
For $\alpha \in A$ such that $\codim_Y Y_\alpha >0$ we define the value of 
$\varphi$ on $Y_\alpha$ by induction on the codimension $\codim_Y Y_\alpha$ of $Y_\alpha$ in $Y$ 
in the following way. 
Suppose that for $k\geq 0$ the values of $\varphi$ on the strata $Y_\alpha \subset Y$ 
such that $\codim_Y Y_\alpha \leq k$ are already determined. 
Then for a stratum $Y_\alpha \subset Y$ such that $\codim_Y Y_\alpha =k+1$ we set 
\begin{equation}\label{eq-def}
\varphi(p_\alpha) \coloneq 
\sum_{\substack{\alpha^\prime \in A \\\codim_Y Y_{\alpha^\prime}\leq k}} 
\varphi(Y_{\alpha^\prime}) \cdot \chi(B(p_\alpha;\e)\cap \{\psi_\alpha >0\}\cap Y_{\alpha^\prime}),
\end{equation}
where $p_\alpha \in Y_\alpha$ is a reference point of the connected stratum $Y_\alpha \subset Y$, 
$B(p_\alpha;\e) \subset X$ is an open ball centered at $p_\alpha$ with radius $0<\e \ll 1$ 
in $X$ and $\psi_\alpha$ is a real analytic function on a neighborhood of $p_\alpha$ in $X$ 
satisfying the conditions: 
\begin{enumerate}
\item[(i)] $\psi_\alpha|_{Y_\alpha}\equiv 0$,
\item[(ii)] $d\psi(p_\alpha) \in T_{Y_\alpha}^\ast X \setminus 
\bigcup_{\alpha^\prime \neq \alpha}\overline{T_{Y_{\alpha^\prime}}^\ast X}$, 
\end{enumerate}
where for the condition (ii) we used the natural identification $T^\ast(X_{\RR})\simeq (T^\ast X)_{\RR}$ 
(see e.g. \cite[Section 11.1]{KS90}). 
In \cite[Chapter 6]{Kas83a} Kashiwara proved that the constructible function 
$\varphi \colon X\longrightarrow \ZZ$ on $X$ such that 
$\varphi|_{X\setminus Y} \equiv 0$ 
thus defined does not depend on the choice of Whitney stratifications of $Y\subset X$. 
In \cite{Mac74} MacPherson defined the same function $\varphi \in \CF_{\ZZ}(X)$ 
by using the Nash blow-up of $Y\subset X$ in order to settle a conjecture of 
Grothendieck and Deligne. 
We call it the Euler obstruction of $Y\subset X$ and denote it by $\mathrm{Eu}_Y \in CF_{\ZZ}(X)$. 
Euler obstructions are defined also for reducible analytic subsets $Y\subset X$ as follows. 
For a reducible analytic subset $Y\subset X$ let $Y=\bigcup_{i\in I}Y_i$ be its 
irreducible decomposition and set $\mathrm{Eu}_Y\coloneq \sum_{i\in I}\mathrm{Eu}_{Y_i}$. 
Then we can easily see that for the regular part $Y_{\mathrm{reg}} \subset Y$ of $Y$ 
we have $\mathrm{Eu}_Y|_{Y_{\mathrm{reg}}} \equiv 1$. 
From now on, we shall recall Kashiwara's index theorem proved in \cite{Kas73} and \cite{Kas83a}. 
Let $\Perv(\CC_X) \subset \BDC_c(X)$ be the full (abelian) subcategory of $\BDC_c(X)$ 
consisting of perverse sheaves on $X$. 
In this paper, we employ the convention that $\CC_X[n] \in \BDC_c(X)$ is a perverse sheaf on $X$. 
Recall that by Kashiwara's constructibility theorem (see e.g. \cite[Theorem 4.6.6]{HTT08}) 
for a holonomic $\SD_X$-module $\SM$ on $X$ and its solution complex 
$\Sol_X(\SM)\coloneq \rhom_{\SD_X}(\SM, \SO_X) \in \BDC(X)$ 
we have $\Sol_X(\SM)[n] \in \Perv(\CC_X)$. 
Moreover, for the (abelian) category $\Modrh(\SD_X)$ of regular holonomic 
$\SD_X$-modules on $X$ Kashiwara proved that we have an equivalence of categories 
\begin{equation}
\Sol_X(\cdot)[n]\colon \Modrh(\SD_X)^{\op} \simto \Perv(\CC_X),
\end{equation}
where op denotes the opposite category. 
We call it the Riemann-Hilbert correspondence. 
The following beautiful result proved in \cite{Kas73} and \cite{Kas83a} is called Kashiwara's 
index theorem. 
Recall that for a holonomic $\SD_X$-module $\SM \in \Modhol(\SD_X)$ 
on $X$ its characteristic variety $\ch \SM \subset T^\ast X$ is a $\CC^\ast$-conic and 
Lagrangian analytic subset and there exists a Whitney stratification $X=\bigsqcup_{\alpha \in A}X_\alpha$ 
of $X$ such that 
\begin{equation}
\ch \SM \subset \bigsqcup_{\alpha \in A}T_{X_\alpha}^\ast X
\end{equation}
(see e.g. \cite[Theorem E.3.9]{HTT08} etc.). 
For $\alpha \in A$ let $\mult_{T_{X_\alpha}^\ast X}(\SM) \geq 0$ be the multiplicity of 
$\SM$ along the Lagrangian submanifold $T_{X_\alpha}^\ast X \subset T^\ast X$. 
\begin{theorem}
In the situation as above, we have an equality 
\begin{equation}
\chi(\Sol_X(\SM))=\sum_{\alpha \in A}(-1)^{\codim_X X_\alpha} 
\mult_{T_{X_\alpha}^\ast X}(\SM) \cdot \mathrm{Eu}_{\overline{X_\alpha}}.
\end{equation}
Equivalently, for the perverse sheaf $F\coloneq \Sol_X(\SM)[n] \in \Perv(\CC_X)$ on $X$ 
we have 
\begin{equation}
\chi(F) =\sum_{\alpha \in A}(-1)^{\dim X_\alpha}\mult_{T_{X_\alpha}^\ast X}
(\SM) \cdot \mathrm{Eu}_{\overline{X_\alpha}}.
\end{equation}
\end{theorem}
On the complex manifold $X$ we can consider also $\RR$-constructible sheaves defined by 
subanalytic stratifications of $X_{\RR}$. 
Let $\BDC_{\RR-c}(X) \subset \BDC(X)$ be the full subcategory of $\BDC(X)=\BDC(X_{\RR})$ 
consisting of $\RR$-constructible sheaves. 
In \cite[Chapter IX]{KS90} for an $\RR$-constructible sheaf $F\in \BDC_{\RR-c}(X)$ 
Kashiwara and Schapira defined a Lagrangian cycle $\CCyc(F)$ in $T^\ast X_{\RR} \simeq (T^\ast X)_{\RR}$. 
We call it the characteristic cycle of $F$. 
Let us recall the standard properties of characteristic cycles. 
First, if $Y\subset X$ is a complex submanifold of $X$ and $k\in \ZZ$, 
then for $\CC_Y[k] \in \BDC_c(X) \subset \BDC_{\RR-c}(X)$ we have $\CCyc(\CC_Y[k]) 
=(-1)^k [T_Y^\ast X]$. 
Moreover, for $F\in \BDC_{\RR-c}(X)$ and its micro-support $\msupp(F) \subset T^\ast X_{\RR} \simeq (T^\ast X)_{\RR}$ 
the support of $\CCyc(F)$ is contained in $\msupp(F)$. 
For a perverse sheaf $F \in \Perv(\CC_X) \subset \BDC_c(X) \subset \BDC_{\RR-c}(X)$ 
on $X$ let $X=\bigsqcup_{\alpha \in A}X_\alpha$ be a (complex analytic) Whitney stratification of $X$ 
such that 
\begin{equation}
\msupp(F) \subset \bigsqcup_{\alpha \in A}T_{X_\alpha}^\ast X.
\end{equation}
Then by \cite[Theorem 9.5.2]{KS85} for any point 
\begin{equation}
p_\alpha \in T_{X_\alpha}^\ast X \setminus \bigcup_{\alpha^\prime \neq \alpha}
\overline{T_{X_{\alpha^\prime}}^\ast X}
\end{equation}
there exists $m_\alpha \geq 0$ such that we have an isomorphism $F\simeq 
\CC_{X_\alpha}^{\oplus m_\alpha}[\dim X_\alpha]$ 
in the localized category $\BDC(X;\{p_\alpha\})$ (for the definition see \cite[Definition 6.1.1]{KS90}). 
This implies that we have $\CCyc(F)=(-1)^{\dim X_\alpha}m_\alpha \cdot [T_{X_\alpha}^\ast X]$ 
on a neighborhood of $p_\alpha \in T^\ast X$ in $T^\ast X$. 
In fact, also the global equality 
\begin{equation}
\CCyc(F) =\sum_{\alpha \in A} (-1)^{\dim X_\alpha}m_\alpha \cdot [T_{X_\alpha}^\ast X]
\end{equation}
on the whole $T^\ast X$ holds true. 
Let $\SM \in \Mod_{\mathrm{rh}}(\SD_X)$ be the regular holonomic $\SD_X$-module which corresponds 
to $F\in \Perv(\CC_X)$ via the Riemann-Hilbert correspondence i.e. there exists an isomorphism 
$\Sol_X(\SM)[n] \simeq F$. 
Then by \cite[Theorem 10.1.1]{KS85} we have $\ch \SM =\msupp(F)$. 
Hence it is very natural to expect that the following proposition holds true. 
Since we do not find its proof in the literature, 
for the convenience of readers we shall give a short proof to it. 
\begin{proposition}\label{prop-mult}
In the situation as above, for any $\alpha \in A$ the multiplicity $\mult_{T_{X_\alpha}^\ast X}(\SM) \geq 0$ 
of $\SM$ along $T_{X_\alpha}^\ast X \subset T^\ast X$ is equal to $m_\alpha \geq 0$.  
\end{proposition}
\begin{proof}
Note that for a point 
\begin{equation}
p_\alpha \in T_{X_\alpha}^\ast X \setminus \bigcup_{\alpha^\prime \neq \alpha} 
\overline{T_{X_{\alpha^\prime}}^\ast X}
\end{equation}
we have an isomorphism $\Sol_X(\SM)[\codim_X X_\alpha] \simeq F[\codim_X 
X_\alpha -n]=F[-\dim X_\alpha] \simeq \CC_{X_\alpha}^{\oplus m_\alpha}$ 
in $\BDC(X;\{p_\alpha\})$. 
By a result of \cite{SKK} we have the vanishing 
\begin{equation}
H^j \mu_{X_\alpha}(\SO_X)=0 \quad (j\neq \codim_X X_\alpha).
\end{equation}
Moreover $\mathcal{C}_{X_\alpha |X}^{\RR} \coloneq H^{\codim_X X_\alpha} \mu_{X_\alpha}(\SO_X)$ 
on $T_{X_\alpha}^\ast X$ is called the sheaf of holomorphic microfunctions along $X_\alpha$. 
Then by \cite[Theorem3.2.1]{Kas83a} we can easily show that there exists an isomorphism 
\begin{equation}
\rhom_{\SD_X}(\SM, \mathcal{C}_{X_\alpha |X}^{\RR})_{p_\alpha} \simeq 
\left(\CC_{T_{X_\alpha}^\ast X}\right)_{p_\alpha}^{\oplus \mult_{T_{X_\alpha}^\ast X}(\SM)}.
\end{equation}
Moreover, there exist isomorphisms 
\begin{align}
\rhom_{\SD_X}(\SM, \mathcal{C}_{X_\alpha |X}^{\RR})_{p_\alpha} &\simeq 
\mu_{X_\alpha}(\rhom_{\SD_X}(\SM, \SO_X)[\codim_X X_\alpha])_{p_\alpha} \\
&\simeq \mu_{X_\alpha}(\CC_{X_\alpha}^{\oplus m_\alpha})_{p_\alpha} \\
&\simeq (\CC_{T_{X_\alpha}^\ast X})_{p_\alpha}^{\oplus m_\alpha}
\end{align}
where in the second (resp. third) isomorphism we used \cite[Corollary 5.4.10(i)]{KS90} 
(resp. \cite[Proposition 4.3.4]{KS90}). 
We thus obtain $\mult_{T_{X_\alpha}^\ast X} (\SM) =m_\alpha$. 
This completes the proof.
\end{proof}

Now we recall the following classical result. 

\begin{lemma}\label{lem-perversity}
Let $F$ be a perverse sheaf on a complex manifold $X$. 
Then for a complex analytic Whitney stratification $\mathcal{S}$ of $X$ the 
following two conditions are equivalent.
\begin{enumerate}
\item[\rm{(i)}] 
$\msupp(F) \subset \bigsqcup_{S\in \mathcal{S}} T_S^\ast X.$
\item[\rm{(ii)}] 
$F\in \Perv(\CC_X) \text{ is adapted to } \mathcal{S}$ i.e. 
$H^jF|_S$ is locally constant for any $S\in \mathcal{S}$ and $j\in \ZZ$.
\end{enumerate}
\end{lemma}

\begin{proof}
By the Riemann-Hilbert correspondence, there exists a regular 
holonomic $\SD_X$-modules $\SM$ such that 
\begin{equation}
\rhom_{\SD_X}(\SM, \SO_X)[\dim X] \simeq F.
\end{equation}
Moreover by \cite[Theorem 10.1.1]{KS85} the characteristic variety $\ch(\SM) \subset T^\ast X$ 
of $\SM$ is equal to $\msupp(F) \subset T^\ast X$. Then by \cite[Proposition 4.6.1]{HTT08} 
we see that the condition (i) implies the one (ii). 
By \cite[Proposition 4.2.3]{BMM94} the converse also holds true. 
\end{proof}

\begin{proposition}\label{prop-adapt}
Let $X$ be a complex manifold and $\SF$ a perverse sheaf on $X\times \CC$. 
Let $\mathcal{S}$ be a complex analytic Whitney stratification of $X\times \CC$ 
adapted to $\SF$ such that $X\times \{0\} \subset X\times \CC$ is a union of some strata 
in it and set $\mathcal{S}_0 \coloneq \{S\in \mathcal{S} 
\mid S\subset X\times \{0\} \} \subset \mathcal{S}$. 
Then for the projection $t\colon X\times \CC \longrightarrow \CC$ the perverse sheaf 
$F\coloneq \psi_t(\SF)[-1] \in \Perv(\CC_X)$ on $X$ is adapted to $\mathcal{S}_0$.
\end{proposition}

\begin{proof}
By lemma \ref{lem-perversity} it suffices to show the inclusion 
\begin{equation}\label{eq-msupp}
\msupp(F) \subset \bigsqcup_{S\in \mathcal{S}_0} T_S^\ast X.
\end{equation}
The problem being local, we may assume that for any $\tau \in \CC^\ast$ such that 
$0<\abs{\tau} \ll 1$ the complex hypersurface $t^{-1}(\tau) =X\times 
\{\tau\} \subset X\times \CC$ 
intersects any stratum $S\in \mathcal{S}$ in $\mathcal{S}$ 
such that $S\subset X\times \CC^\ast \ (\Longleftrightarrow S\in \mathcal{S}\setminus \mathcal{S}_0)$ 
transversally. 
Then for $\tau \in \CC^\ast$ such that $0<\abs{\tau} \ll 1$ 
we obtain 
\begin{equation}
\msupp(\SF|_{X\times \{\tau\}})  \subset \bigsqcup_{S\in 
\mathcal{S}\setminus \mathcal{S}_0} T_{S\cap (X\times \{\tau\})}^\ast (X\times \{\tau\}).
\end{equation}
Since by \cite[Theorem 4.7]{FKT26} we have 
\begin{equation}
\CCyc(\psi_t(\SF)) =\lim_{\tau \to 0}\CCyc(\SF|_{X\times \{\tau\}}),
\end{equation}
this implies that 
\begin{align}
\msupp(F)&=\supp \CCyc(\psi_t(\SF)) \subset \overline{\bigcup_{\tau 
\in \CC^\ast, 0<\abs{\tau} \ll 1}  
\msupp(\SF|_{X\times \{\tau\}}) \times \{\tau\} }\\ 
&\subset 
\bigsqcup_{S\in \mathcal{S}\setminus \mathcal{S}_0} 
\overline{\bigcup_{\tau \in \CC^\ast, 0<\abs{\tau} \ll 1}
T_{S\cap (X\times \{\tau\})}^\ast (X\times \{\tau\} ) \times \{\tau\} },
\end{align}
where the closures are taken in the relative cotangent bundle 
$T^\ast(X\times \CC/\CC) \simeq T^\ast X \times \CC$ 
associated to the projection $t\colon X\times \CC \longrightarrow \CC$. 
On the other hand, by \cite[Th\'eor\`eme 4.2.1]{BMM94} for any 
$S\in \mathcal{S} \setminus \mathcal{S}_0$ 
there exists an inclusion 
\begin{equation}
t^{-1}(0) \times_{X\times \CC} T^\ast (X\times \CC) \cap 
\overline{\bigcup_{\tau \in \CC^\ast, 0<\abs{\tau} \ll 1}T_{S\cap (X\times \{\tau\})}^\ast 
(X\times \CC )} \subset \bigsqcup_{S\in \mathcal{S}_0} T_S^\ast (X\times \CC).   
\end{equation}
Then we obtain the desired inclusion (\ref{eq-msupp}).
\end{proof}

\section{Some auxiliary results for constructible sheaves}\label{sec-pre}
In this section, we prove some formulas 
for real and complex constructible sheaves, 
which will be used to describe their stalks 
and hyper cohomology groups in subsequent sections.
Let $X$ and $Y$ be real analytic manifolds and 
$F\in \BDC_{\RR -c}(X\times Y)$ an $\RR$-constructible sheaf 
on $X\times Y$.
For a point $(x,y)\in X\times Y$ and $0<\e, \delta \ll 1$ 
let $B(x;\e)\subset X$ (resp. $B(y;\delta)\subset Y$) be 
the open ball in $X$ (resp. $Y$) centered at $x\in X$ 
(resp. $y\in Y$) with radius $\e >0$ (resp. $\delta >0$).
Then our first objective here is to describe the set of 
points $(\e, \delta)\in \RR^2$ ($0< \e, \delta \ll 1$) for which 
we have an isomorphism 
\begin{equation}
\rsect (B(x;\e)\times B(y;\delta); F) \simto F_{(x,y)}.
\end{equation}
The problem being local, we may assume that $X$ and $Y$ are 
$\RR^n$ and $\RR^m$ respectively and $x=0 \in X=\RR^n$, 
$y=0\in Y=\RR^m$.
We may assume also that the support of $F$ is compact.
Let $\mathcal{S}$ be a subanalytic Whitney stratification 
of $X\times Y=\RR^{n+m}$ adapted to $F\in \BDC_{\RR -c}(X\times Y)$
and define a real analytic function 
$\varphi \colon X\times Y \longrightarrow \RR$ 
(resp. $\psi \colon X\times Y \longrightarrow \RR$) 
by $\varphi (x,y)\coloneq \| x\|^2=x_1^2+\cdots +x_n^2$ 
(resp. $\psi (x,y)\coloneq \| y\|^2=y_1^2+\cdots +y_m^2$).
Then by the microlocal Bertini-Sard theorem \cite[Proposition 8.3.12]{KS90} 
there exists $0<e_0 \ll 1$ such that for any $0<e\leq e_0$ and 
$S \in \mathcal{S}$ we have the transversality 
\begin{equation}
\varphi^{-1}(e) \pitchfork S.
\end{equation}
Using this transversality, for each $0<e\leq e_0$ 
we obtain naturally a subanalytic Whitney stratification 
of $\varphi^{-1}([0,e])=\{ (x,y)\in X\times Y \mid \|x\|^2 \leq e\}
=\overline{B(0;\sqrt{e})} \times Y \subset X\times Y$ adapted to 
$F_{\varphi^{-1}([0,e])} \in \BDC_{\RR -c}(X\times Y)$
and denote it by $\mathcal{S}_e$.
For $0<e\leq e_0$ we define a closed conic subanalytic Lagrangian 
subset $\Lambda_e \subset T^\ast (X\times Y)$ of $T^\ast(X\times Y)$
by 
\begin{equation}
\Lambda_e \coloneq \bigsqcup_{S\in \mathcal{S}_e}T_S^\ast(X\times Y) \quad \subset T^\ast(X\times Y).
\end{equation}
Then by applying the microlocal Bertini-Sard theorem 
\cite[Proposition 8.3.12]{KS90} to $\Lambda_e \subset T^\ast(X\times Y)$ 
and the real analytic function $\psi \colon X\times Y \longrightarrow
\RR$ we obtain a discrete subanalytic subset $D_e \subset \RR_{>0}$ 
such that for $d> 0$ we have an equivalence 
\begin{equation}
d\not \in  D_e \quad \Longleftrightarrow \quad \psi^{-1}(d) \pitchfork S \quad 
\text{for any $S\in \mathcal{S}_e$}.
\end{equation}
This in particular implies that for any $0<e\leq e_0$ and 
$d>0$ such that $d\not \in D_e$ the submanifold 
\begin{equation}
\varphi^{-1}(e) \cap \psi^{-1}(d) =\partial B(0;\sqrt{e}) \times 
\partial B(0;\sqrt{d}) \quad \subset X \times Y
\end{equation}
of $X\times Y$ intersects any stratum $S\in \mathcal{S}$ transversally.
Now we define a subset $D\subset \RR^2$ of $\RR^2$ by 
\begin{equation}
D\coloneq \bigsqcup_{0<e\leq e_0}(\{e\} \times D_e) \quad \subset \RR^2.
\end{equation}
\begin{lemma}\label{lem-subanl}
The subset $D\subset \RR^2$ is subanalytic in $\RR^2$ and 
has dimension $\leq 1$.
\end{lemma}
\begin{proof}
We first define a locally closed subanalytic subset $K\subset X\times Y\times \RR$ 
of $X\times Y\times \RR$ by 
\begin{equation}
K\coloneq \{ (x,y,e) \mid 0<e\leq e_0, \quad \varphi(x,y) =\|x\|^2 \leq e\}.
\end{equation}
Note that for $0<e\leq e_0$ we have 
\begin{equation}
K\cap (X\times Y\times \{e\}) \simeq \overline{B(0;\sqrt{e})} \times Y.
\end{equation}
Next for a stratum $S\in \mathcal{S}$ in $\mathcal{S}$ we set 
\begin{equation}
\begin{cases}
\widetilde{S_1} \coloneq (S\times \RR_{>0}) \cap \mathrm{Int} K, \\
\widetilde{S_2} \coloneq (S\times \RR_{>0}) \cap \partial K.
\end{cases}
\end{equation}
Then $\widetilde{S_1}, \widetilde{S_2} \subset X\times Y\times \RR$
are locally closed subanalytic submanifolds of $X\times Y\times \RR$.
Let $p\colon X\times Y\times \RR \longrightarrow \RR$ 
be the projection and $T^\ast(X\times Y\times \RR /\RR) \simeq T^\ast(X\times Y) \times \RR$ 
the relative cotangent bundle associated to it.
For $S\in \mathcal{S}$ we then define subsets 
$\Xi_{\widetilde{S_i}} \subset T^\ast(X\times Y\times \RR/\RR)$ 
($i=1,2$) in it by
\begin{equation}
\Xi_{\widetilde{S_i}} \coloneq \bigsqcup_{0<e\leq e_0} \left[ T_{\widetilde{S_i} \cap(X\times Y\times \{e\})}^\ast (X\times Y\times \{e\}) \times \{e\}\right]
\quad \subset T^\ast(X\times Y\times \RR/\RR).
\end{equation}
Note that for $S\in \mathcal{S}$ and $0<e\leq e_0$ we have 
\begin{equation}
\widetilde{S_i} \cap (X\times Y\times \{e\}) \simeq 
\begin{cases}
S\cap B(0;\sqrt{e}) \quad & (i=1), \\
S\cap \partial B(0;\sqrt{e}) \quad & (i=2)
\end{cases}
\end{equation}
and hence for $0<e\leq e_0$ the subset 
\begin{equation}
\bigsqcup_{S\in \mathcal{S}} \bigsqcup_{i=1,2} T_{\widetilde{S_i}\cap (X\times Y\times \{e\})}(X\times Y\times \{e\}) \quad \subset T^\ast(X\times Y\times \{e\})
\end{equation}
of $T^\ast(X\times Y\times \{e\}) \simeq T^\ast(X\times Y)$ 
is nothing but the closed conic subanalytic Lagrangian subset 
$\Lambda_e \subset T^\ast(X\times Y)$.
This implies that 
\begin{equation}
\Xi \coloneq \bigsqcup_{S\in \mathcal{S}} \bigsqcup_{i=1,2} \Xi_{\widetilde{S_i}} 
= \bigsqcup_{0<e\leq e_0}(\Lambda_e \times \{e\}) \quad \subset T^\ast(X\times Y\times \RR /\RR)
\end{equation}
is a family of closed conic subanalytic Lagrangian subsets 
$\Lambda_e \subset T^\ast(X\times Y)$ 
parametrized by $0<e\leq e_0$.
By the proof of \cite[Lemma 4.6]{FKT26} in \cite[Appendix B]{FKT26}, 
we see that $\Xi \subset T^\ast(X\times Y\times \RR/\RR)$ 
is subanalytic in $T^\ast(X\times Y\times \RR/\RR)$.
We define also a real analytic submanifold $\Xi^\prime \subset 
T^\ast(X\times Y\times \RR/\RR) \simeq T^\ast(X\times Y) \times \RR$ 
by 
\begin{equation}
\Xi^\prime \coloneq \{ (x,y,d\psi(x,y)) \mid (x,y)\in X\times Y\} \times \RR \simeq X\times Y\times \RR.
\end{equation} 
Let $\pi \colon T^\ast(X\times Y\times \RR/\RR) \simeq T^\ast(X\times Y) \times \RR 
\longrightarrow X\times Y$
be the projection, set $q\coloneq \psi \circ \pi \colon T^\ast(X\times Y\times \RR/\RR) \longrightarrow \RR$ 
and consider the morphism $\Phi \coloneq (p,q)\colon T^\ast(X\times Y\times \RR/\RR) \longrightarrow \RR^2$
of real analytic manifolds defined by it. 
Then it is easy to see that $\Phi$ is proper on $\overline{\Xi \cap \Xi^\prime} 
\subset T^\ast(X\times Y \times \RR/\RR)$ and 
\begin{equation}
\Phi (\Xi \cap \Xi^\prime) = D =\bigsqcup_{0<e\leq e_0}(\{ e\}\times D_e) \quad \subset \RR^2.
\end{equation}
We thus see that $D$ is subanalytic in $\RR^2$.
As $D_e\subset \RR$ is discrete for any $0<e\leq e_0$, 
it is also clear that $\dim{D} \leq1$.
\end{proof}
By the map $\Psi \colon \RR^2 \longrightarrow \RR^2 ((\e, \delta) \mapsto (\e^2, \delta^2))$
we define a subset $E\subset \RR^2$ by 
\begin{equation}
E\coloneq \Psi^{-1}(D) \cap \RR_{>0}^2 \quad \subset \RR^2.
\end{equation}
We set $\e_0 \coloneq \sqrt{e_0}$. 
Then by Lemma \ref{lem-subanl} we see that $E$ is subanalytic 
in $\RR^2$ and has dimension $\leq 1$.
Moreover, by our construction of $E$, for any $0<\e \leq \e_0$ 
and $\delta >0$ such that $(\e, \delta) \not \in E$ 
the three submanifolds 
\begin{equation}
\partial B(0;\e) \times \partial B(0;\delta), \quad 
\partial B(0;\e) \times  B(0;\delta), \quad 
B(0;\e) \times \partial B(0;\delta)
\end{equation}
of $X\times Y$ intersects any stratum $S\in \mathcal{S}$ in 
$\mathcal{S}$ transversally. 
\begin{lemma}\label{lem-dir}
There exist $0< \e_1 \ll \e_0$ and $\lambda>0$ such that 
for the subset $\ell_\lambda \coloneq \{ (\e, \lambda \e) \mid 0<\e \leq \e_1\} \subset \RR_{>0}^2$
we have $\ell_\lambda \cap E = \varnothing$.
\end{lemma}
\begin{proof}
Let $C_{\{0\}}(E) \subset T_{\{0\}} \RR^2 \simeq \RR^2$ be the 
normal cone of $E\subset \RR^2$ along the submanifold $\{0\} \subset \RR^2$.
Then $C_{\{0\}}(E)$ is subanalytic and has dimension $\leq 1$. 
From this the assertion immediately follows.
\end{proof}
By this lemma, the non-characteristic deformation lemma \cite[Proposition 2.7.2]{KS90} 
and \cite[Corollary 5.4.9]{KS90} we thus obtain the following result.
\begin{proposition}\label{prop-stalk}
Let $0<\e_1 \ll \e_0$ and $\lambda>0$ be as in Lemma \ref{lem-dir}. 
Then for any $0<\e \leq \e_1$ we have an isomorphism
\begin{equation}
\rsect(B(0;\e)\times B(0;\lambda \e);F) \simeq F_{(0,0)}.
\end{equation}
\end{proposition}
From now, we shall apply Proposition \ref{prop-stalk} to nearby cycle sheaves. 
For this purpose, we assume that $X$ and $Y$ are complex manifolds 
and take a complex constructible sheaf $G\in \BDC_c(X\times Y\times \CC)$ 
on $X\times Y\times \CC$. 
Let $t\colon X\times Y\times \CC \longrightarrow \CC$ be the projection. 
Then our objective is to describe the stalks of the nearby cycle sheaf 
$\psi_t(G) \in \BDC_c(X\times Y)$ on $X\times Y$ at points $(x,y)\in X\times Y$. 
The problem being local, we may assume that $X$ and $Y$ are $\CC^n$ and $\CC^m$ 
respectively and $x=0\in X=\CC^n$, $y=0\in Y=\CC^m$. 
Then we obtain the following result. 
\begin{proposition}\label{prop-cont}
There exist $0<\e_1 \ll 1$ and $\lambda>0$ such that 
we have an isomorphism 
\begin{equation}
\psi_t(G)_{(0,0)} \simeq \rsect(B(0;\e)\times B(0;\lambda \e)\times \{t\};G)
\end{equation}
for any $0<\e \leq \e_1$ and $t\in \CC^\ast$ such that $0<\abs{t} \ll 1$.
\end{proposition} 
\begin{proof}
The proof is similar to \cite[Theorem 2.6]{Tak25}. 
Let $\mathcal{S}$ be a (complex analytic) Whitney stratification of $X\times Y \times \CC$ 
adapted to $G\in \BDC_c(X\times Y\times \CC)$ such that $X\times Y\times \{0\} \subset X\times Y\times \CC$ 
is a union of some strata in it. 
Then we obtain a Whitney stratification $\mathcal{S}_0 \coloneq \{S\in \mathcal{S} \mid S\subset X\times Y\times \{0\} \}$ 
of $X\times Y\times \{0\} \simeq X\times Y$. 
Let $\mathcal{T}_0$ be a Whitney stratification of $X\times Y$ adapted to 
$\psi_t(G) \in \BDC_c(X\times Y)$ and define a Whitney stratification $\mathcal{T}$ 
of $X\times Y\times \CC$ by $\mathcal{T}\coloneq \mathcal{T}_0 \sqcup \{ X\times Y\times \CC^\ast\}$. 
We take a refinement $\mathcal{W}$ of $\mathcal{S}$ and $\mathcal{T}$ and 
set $\mathcal{W}_0 \coloneq \{ S\in \mathcal{W} \mid S\subset X\times Y\times \{0\} \}$. 
Then by applying our arguments in this section to the Whitney stratification $\mathcal{W}_0$ 
of $X\times Y\times \{0\} \simeq X\times Y$, we see that there exist 
$0<\e_1 \ll 1$ and $\lambda>0$ such that for any $0<\e \leq \e_1$ the three submanifolds 
\begin{equation}
\partial B(0;\e)\times \partial B(0;\lambda \e), \quad 
\partial B(0;\e)\times  B(0;\lambda \e), \quad
B(0;\e)\times \partial B(0;\lambda \e)
\end{equation}
of $X\times Y$ intersect any stratum $S\in \mathcal{W}_0$ in $\mathcal{W}_0$ transversally. 
In particular, for any $0<\e \leq \e_1$ we obtain an isomorphism 
\begin{equation}
\rsect(B(0;\e)\times B(0;\lambda \e); \psi_t(G)) \simto \psi_t(G)_{(0,0)}.
\end{equation}
Note that we have $t^{-1}(0)=X\times Y\times \{0\}$ and 
by \cite[Th\'{e}or\`{e}me 4.2.1]{BMM94} the Whitney stratification $\mathcal{W}$ 
of $X\times Y\times \CC$ satisfies Thom's condition $a_t$ for the holomorphic 
function $t\colon X\times Y\times \CC \longrightarrow \CC$. 
Then with the help of the above transversality we can modify the proof 
of \cite[Theorem 2.6]{Tak25} to obtain the assertion.
\end{proof}
Now let $f\colon X\times Y \longrightarrow \CC$ be a non-constant holomorphic 
function on $X\times Y=\CC^{n+m}$ such that $f(0,0)=0$ and 
$F\in \BDC_c(X\times Y)$ a complex constructible sheaf on $X\times Y$. 
Let $i_f\colon X\times Y \longhookrightarrow X\times Y\times \CC$ ($(x,y)\longmapsto (x,y,f(x,y))$) 
be the graph embedding associated to $f$. 
Then by \cite[Theorem 2.5]{Tak25} there exists an isomorphism 
\begin{equation}
\psi_t({i_f}_\ast(F)) \simeq \iota_\ast \psi_f(F),
\end{equation}
where $\iota \colon f^{-1}(0) \longhookrightarrow X\times Y$ is the inclusion map. 
As a very special case of Proposition \ref{prop-stalk}, 
we thus obtain the following result.
\begin{corollary}
In the situation as above, there exists $0<\e_1 \ll 1$ and $\lambda >0$ such that 
we have an isomorphism 
\begin{equation}
\psi_f(F)_{(0,0)} \simeq \rsect (B(0;\e)\times B(0;\lambda \e) \cap f^{-1}(t); F)
\end{equation}
for any $0<\e \leq \e_1$ and $t\in \CC^\ast$ such that $0<\abs{t} \ll 1$. 
\end{corollary}

\section{Verdier specializations and Milnor fibers}\label{sec-verdier}
In this section, for a complete intersection subvariety 
$Y$ of $X=\CC^n$ (for the definition, 
see e.g. \cite{Ful98} and \cite{Loo84}) 
we study the relation between the Verdier specialization of 
the constant sheaf $\CC_X$ on $X$ along $Y$ 
and the Milnor fibers of $Y\subset X=\CC^n$. 
Recall that when $d\coloneq \mathrm{codim}_X Y\geq 2$ 
we do not always have the notion of Milnor fibers of $Y\subset X =\CC^n$. 
For this reason, we assume here some Thom condition to 
describe the structure of the Verdier specialization. 
Let $Y=\{f_1=\cdots =f_d=0\} \subset X=\CC^n$ ($f_i(x)\in \CC[x_1, \ldots ,x_n]$) 
be a complete intersection subvariety of $X=\CC^n$ such that $0\in Y$ 
and 
\begin{equation}
i_f\colon X\longhookrightarrow X\times \CC_s^d 
\quad (x\longmapsto (x, f_1(x), \ldots , f_d(x)))
\end{equation}
the graph embedding of $f=(f_1, \ldots ,f_d) \colon X \longrightarrow \CC_s^d$. 
Then we obtain a morphism of cones 
\begin{equation}
\mathrm{C} (i_f)\colon \mathrm{C}_Y X \longrightarrow 
\mathrm{C}_{X\times \{0\}}(X\times \CC^d) \simeq X\times \CC^d 
\end{equation}
induced by $i_f\colon (X,Y) \longrightarrow (X\times \CC^d, X\times \{0\})$. 
As $Y=\{f_1=\cdots =f_d=0\} \subset X=\CC^n$ is C.I., 
the image of the injective morphism $\mathrm{C}(i_f)$ is equal to 
\begin{equation}
(Y\times \{0\})\times_{X\times \{0\}} 
\mathrm{C}_{X\times \{0\}}(X\times \CC^d) \simeq Y\times \CC^d. 
\end{equation}
For a complex constructible sheaf $F\in \BDC_c(X)$ on $X=\CC^n$ by 
Verdier \cite{Ver83} we thus obtain an isomorphism 
\begin{equation}
\mathrm{Sp}_{X\times \{0\} \mid X\times \CC^d} ({i_f}_\ast F) 
\simeq \mathrm{C}(i_f)_\ast \mathrm{Sp}_{Y\mid X}(F). 
\end{equation}
Hence, for the study of $\mathrm{Sp}_{Y\mid X}(F)\in \BDC_c(\mathrm{C}_Y X)$ 
it suffices to study $\mathrm{Sp}_{X\times \{0\} \mid 
X\times \CC^d} ({i_f}_\ast F) \in \BDC_c(X\times \CC^d)$. 
For the C.I. subvariety $Y=\{f_1=\cdots =f_d=0\} \subset X=\CC^n$ 
we denote its discriminant set $f(\mathrm{Sing} f)\subset \CC^d$ 
by $\Delta_f$. 
We know that it is contained in a complex hypersurface in $\CC^d$. 
Moreover, we can easily see that the normal cone 
$\mathrm{C}_{\{0\}}(\Delta_f)$ of $\Delta_f$ along $\{0\} \subset \CC^d$ 
is a $\CC^\ast$-conic analytic subset of $\mathrm{C}_{\{0\}}(\CC^d) \simeq \CC^d$ 
such that $\mathrm{C}_{\{0\}}(\Delta_f)\neq \mathrm{C}_{\{0\}}(\CC^d)$. 

\begin{definition}\label{def-thom}
For a point $a\in \CC^d \setminus \mathrm{C}_{\{0\}}(\Delta_f)$ 
we say that $f=(f_1,\ldots ,f_d)\colon X\longrightarrow \CC_s^d$ 
satisfies the Thom condition in the derection $a$ if there exist 
a neighborhood $U\subset \CC^d \setminus \mathrm{C}_{\{0\}}(\Delta_f)$ 
of $a$ in $\CC^d \setminus \mathrm{C}_{\{0\}}(\Delta_f)$ and 
a Whitney stratification $Y=\bigsqcup_{\alpha \in A} Y_{\alpha}$ 
of $Y$ such that for any sequence of points $p_i\in \CC^\ast U 
\subset \CC^d \setminus \mathrm{C}_{\{0\}}(\Delta_f)$ 
($i=1,2,3, \ldots $) in the open cone $\CC^\ast U\subset \CC^d$ and 
points $q_i\in f^{-1}(p_i)$ ($i=1,2,3, \ldots $) satisfying the conditions 
\begin{equation}
\begin{cases}
q_i \longrightarrow ^\exists q \in Y_{\alpha} \subset Y \quad 
(i\longrightarrow +\infty), \\
T_{q_i}(f^{-1}(f(q_i)))=T_{q_i}(f^{-1}(p_i)) \longrightarrow 
^\exists \mathbf{T} (\simeq \CC^{n-d})\subset T_qX \quad 
(i\longrightarrow +\infty )
\end{cases}
\end{equation}
we have $\mathbf{T} \supset T_qY_{\alpha}$.
\end{definition}
By this definition, as in \cite[Theorem 1.1 (1)]{CMSS16} 
we obtain the following result.
\begin{lemma}\label{lem-fib}
Assume that $f=(f_1,\ldots , f_d)\colon X\longrightarrow \CC_s^d$ 
satisfies the Thom condition in a direction $a\in \CC^d \setminus \mathrm{C}_{\{0\}}(\Delta_f)$. 
Then there exists a neighborhood $U\subset \CC^d \setminus \mathrm{C}_{\{0\}}(\Delta_f)$ 
of $a$ and $0<\e_0 \ll 1$ such that for any 
$0<\eta \ll \e \leq \e_0$ the restriction 
\begin{equation}
B(0;\e) \cap f^{-1}(D_{\eta}^\ast \cap \CC^\ast U) \longrightarrow D_{\eta}^\ast \cap \CC^\ast U 
\end{equation}
of $f$ is a locally trivial fibration, where we set 
\begin{equation}
D_{\eta}^\ast \coloneq \{s=(s_1, \ldots , s_d) \in \CC^d \mid 0< 
\sqrt{\abs{s_1}^2+ \cdots +\abs{s_d}^2} <\eta \}.
\end{equation}
\end{lemma}
We call the fiber of the locally trivial fibration in Lemma \ref{lem-fib} 
the Milnor fiber of $f=(f_1,\ldots , f_d)\colon X \longrightarrow \CC_s^d$ or of $Y$ 
at the origin $0\in Y\subset X=\CC^n$ in the direction $a\in \CC^d \setminus \mathrm{C}_{\{0\}}(\Delta_f)$ 
and denote it by $M_{Y,0} \subset X\setminus Y$ for short. 
Similarly, for any point $q\in Y$ we can define the Milnor fiber $M_{Y,q}\subset X\setminus Y$ 
of $f=(f_1,\ldots , f_d)\colon X \longrightarrow \CC_s^d$ or of $Y$ at it (in the direction 
$a\in \CC^d \setminus \mathrm{C}_{\{0\}}(\Delta_f)$). 

The following famous example due to L\^e D. T. is a holomorphic map which does 
not satisfy the Thom condition on the whole $\CC^d \setminus \mathrm{C}_{\{0\}}(\Delta_f)$ 
(see e.g. \cite[page 23]{Sea06}).  
We will show that however it satisfies the Thom condition in a direction. 

\begin{example}
Let us consider the holomorphic map $f\colon \CC_{x,y,z}^3 \longrightarrow \CC_{\alpha, \beta}^2$
defined by $(x,y,z)\longmapsto (x^2-y^2z, y)$ and set $Y\coloneq f^{-1}(0)=
\{ (0,0,z) \mid z\in \CC\}\subset X\coloneq \CC_{x,y,z}^3$. 
Since $Y$ is smooth (but all the points in $Y$ are critical points of $f$), 
it has the trivial Whitney stratification. 
It is easy to see that $\Delta_f=\{(0,0)\}$ and hence $\CC^2 \setminus 
C_{\{0\}}(\Delta_f)=\CC^2\setminus \{(0,0)\}$. 
For the point $a\coloneq (1,0) \in \CC^2 \setminus C_{\{0\}}(\Delta_f)=\CC^2\setminus 
\{(0,0)\}$ and the open subset $U\coloneq \{(p,q) \mid \abs q <\abs p\} \subset 
\CC^2 \setminus C_{\{0\}}(\Delta_f)=\CC^2\setminus \{(0,0)\}$ containing it 
we can show that the condition in Definition \ref{def-thom} is satisfied 
as follows. Suppose that we are given 
two sequences of points $p_i\coloneq(\alpha_i, \beta_i) \in U \ (i=1,2,3,\ldots)$ and 
$q_i\coloneq (x_i, y_i,z_i) \in f^{-1}(p_i) \ (i=1,2,3,\ldots)$ such that 
\begin{equation}
q_i=(x_i,y_i,z_i) \longrightarrow ^\exists(0,0, z_0)\in Y \ (\ \text{i.e.} \ p_i=(\alpha_i, \beta_i) 
\longrightarrow (0,0)) \ \text{as} \ i\longrightarrow +\infty.  
\end{equation}
Then the condition $f(q_i)=(x_i^2-y_i^2z_i,y_i)= p_i= (\alpha_i, \beta_i)$ implies that 
\begin{equation}\label{eq-parabola}
y_i=\beta_i \ \text{and} \  \beta_i^2 z_i= x_i^2-\alpha_i \quad (i=1,2,3,\ldots). 
\end{equation}
Moreover we have 
\begin{equation}
f^{-1}(p_i)= 
\begin{cases}
\{ (x, \beta_i, \frac{x^2-\alpha_i}{\beta_i^2} )  \ | \ x \in \CC \}  \quad ( \beta_i \neq 0), \\
\\
\{ ( \pm \alpha_i^{1/2} , 0, z)  \ | \ z \in \CC \}   \quad ( \beta_i =0).
\end{cases}
\end{equation}
Using the condition $|\beta_i| <|\alpha_i|$, 
we thus can see that the algebraic curve $f^{-1}(p_i) \subset X= \CC_{x,y,z}^3$ 
collapses to the $z$-axis in $X= \CC_{x,y,z}^3$ as $p_i=(\alpha_i,\beta_i)$ approaches 
the origin $0=(0,0) \in \CC^2$ in the open subset 
$U\subset \CC^2 \setminus C_{\{0\}}(\Delta_f)=\CC^2\setminus \{(0,0)\}$. 
Let us explain this phenomenon more precisely by describing their tangent spaces. 
First note that since the convergent sequence $z_i$ is bounded and $ |\beta_i| <|\alpha_i|$ we have 
\begin{equation}
\left|\frac{\beta_i^2 z_i}{\alpha_i}\right| \leq  |\alpha_i| |z_i| 
\longrightarrow 0 \quad (i\longrightarrow +\infty).
\end{equation}
Then together with \eqref{eq-parabola} we obtain 
\begin{equation}\label{eq-conv}
\left|\frac{x_i}{\alpha_i^{1/2}}\right| =\left|  
1+\frac{\beta_i^2z_i}{\alpha_i}\right|^{1/2}\longrightarrow 1 \quad (i\longrightarrow +\infty)
\end{equation}
and hence $x_i \neq 0$ for $i\gg 1$. We thus obtain 
\begin{equation}
\left|\frac{y_i^2}{x_i}\right|=\left|\frac{\beta_i^2}{\alpha_i^{1/2}} \right| \left| 
\frac{\alpha_i^{1/2}}{x_i}\right|\leq \left|\frac{\alpha_i^2}{\alpha_i^{1/2}} 
\right| \left| \frac{\alpha_i^{1/2}}{x_i}\right|\longrightarrow 0 \quad (i\longrightarrow +\infty).
\end{equation}
Now since for $i\gg 1$ 
the tangent space of the curve $f^{-1}(p_i)$ at $q_i$ is equal to 
$\mathrm{Span}_\CC\{ \frac{y_i^2}{2x_i}\partial_x+ \partial_z\}$, 
we have 
\begin{equation}
\lim_{i\to +\infty} T_{q_i}f^{-1}(p_i) 
=\mathrm{Span}_\CC\{ \partial_z\} =T_{(0,0,z_0)}Y.
\end{equation}
This means that the Thom condition holds true in the direction 
$a=(1,0)$. However, we can also show that for any Whitney stratification 
$Y=\bigsqcup_{\alpha \in A}Y_\alpha$ 
of $Y$ the Thom condition does not hold true in the direction $b\coloneq (0,1) 
\in \CC^2 \setminus C_{\{0\}}(\Delta_f)=\CC^2\setminus \{(0,0)\}$. 
In order to see this, first let us take a $1$-dimensional stratum 
$Y_\alpha$ and fix a point $q=(0,0,z_1) \in Y_\alpha$.
For $N\gg 1$ let us define a sequence of points $q_t \in X=\CC_{x,y,z}^3$ by
$q_t\coloneq (t^N, t, z_1) \in X=\CC_{x,y,z}^3 \ (t \in\RR\setminus\{0\})$ and consider 
a sequence of points $p_t\coloneq f(q_t)=(t^{2N}-t^2z_1,t) \in \CC_{p,q}^2 \ (t\in \RR\setminus\{0\})$. 
Then it follows that for any neighborhood $V \subset \CC^2 \setminus C_{\{0\}}(
\Delta_f)=\CC^2\setminus \{(0,0)\}$ of $b$ we have
$p_t \in \CC^\ast V \ (0<\abs t\ll 1)$, $q_t \longrightarrow q \ (t\longrightarrow 0)$ 
and $T_{q_t} f^{-1}(p_t)=\mathrm{Span}_\CC \{\partial_x+ 2t^{N-2}\partial_z\}$.
Therefore, we obtain 
\begin{equation}
\lim_{t\to 0} T_{q_t}f^{-1}(p_t) =\mathrm{Span}_\CC\{\partial_x\} \not
 \supset T_q Y_\alpha =\mathrm{Span}_\CC \{\partial_z\}.
\end{equation}
\end{example}
To construct various examples of $f\colon X\longrightarrow \CC^d $ 
satisfying the Thom condition in a direction $a\in \CC^d \setminus C_{\{0\}}(\Delta_f)$, 
we recall the following basic definition. 
\begin{definition}\label{def-transversality}
For complex submanifolds $S_1, S_2, \ldots , S_d \subset Z$ 
of a complex manifold $Z$, 
we say that they intersect transversally if for any point $q\in S_1\cap \cdots \cap S_d$ 
the $\CC$-linear subspace 
\begin{equation}
(T_{S_1}^\ast Z)_q+ \cdots +(T_{S_d}^\ast Z)_q \quad \subset T_q^\ast Z
\end{equation}
of $T_q^\ast Z$ is the direct sum $(T_{S_1}^\ast Z)_q\oplus \cdots \oplus (T_{S_d}^\ast Z)_q$.
\end{definition}
In the case of Definition \ref{def-transversality}, we write 
$S_1 \pitchfork \cdots \pitchfork S_d$. 
Moreover then we can easily see that $S_1\cap \cdots \cap S_d \subset Z$ is a complex 
submanifold of $Z$ and its codimension $\codim_Z (S_1\cap \cdots \cap S_d)$ 
in $Z$ is equal to $\codim_Z S_1 +\cdots +\codim_Z S_d$. 
\begin{definition}\label{def-str-transversality}
We say that the complex hypersurfaces $\{f_i=0\} \subset X \ (1\leq i\leq d)$ 
of $X=\CC^n$ intersect stratified transversally outside the origin $0\in X=\CC^n$ 
if there exist their Whitney stratifications $\mathcal{S}_i \ (1\leq i\leq d)$ 
such that for any $1\leq i\leq d$ the origin $\{0\}\subset \{f_i=0\}$ 
is a stratum in $\mathcal{S}_i$ and for any $d$-tuple $(S_1,\ldots , S_d) 
\ (S_i \in \mathcal{S}_i, S_i \neq \{0\})$ 
of their strata we have $S_1\pitchfork \cdots \pitchfork S_d$ in $X=\CC^d$.
\end{definition}
Now let us set 
\begin{equation}
\Gamma \coloneq \{s=(s_1,\ldots , s_d)\in \CC^d \mid s_1 \cdots s_d=0\} \quad \subset \CC^d.
\end{equation}
Then we have the following result. 
\begin{proposition}\label{prop-af}
Assume that the complex hypersurfaces $\{f_i=0\}\subset X \ (1\leq i \leq d)$ of $X=\CC^n$ 
intersect stratified transversally outside the origin $0\in X=\CC^n$ 
and let $\mathcal{S}_i \ (1\leq i\leq d)$ be their 
Whitney stratfications for which the conditions in Definition \ref{def-str-transversality} are 
satisfied. Then we have 
\begin{enumerate}
\item [(i)] 
The natural partition 
\begin{equation}
\mathcal{P} =\{S_1\cap \cdots \cap S_d\mid S_i \in \mathcal{S}_i \ (1\leq i\leq d) \} 
\end{equation}
of the C.I. $Y=\{f_1=\cdots =f_d=0\} =\{f_1=0\} \cap \cdots \cap \{f_d=0\}$ is a Whitney 
stratification.
\item [(ii)]
For any $a\in \CC^d \setminus (C_{\{0\}}(\Delta_f) \cup \Gamma)$ the morphism 
$f=(f_1,\ldots ,f_d)\colon X\longrightarrow \CC^d$ satisfies the Thom condition in the 
direction $a$ with respect to the Whitney stratification of $Y$ in (i).
\end{enumerate}
\end{proposition}

\begin{proof}
For $1\leq i\leq d$ we define a Whitney stratification $\widetilde{\mathcal{S}_i}$ 
of $X=\CC^n$ by $\widetilde{\mathcal{S}_i}\coloneq \mathcal{S}_i \sqcup \{X\setminus \{f_i=0\}\}$. 
Then we obtain a partition $\widetilde{\mathcal{P}}$ of $X=\CC^n$ defined by 
\begin{equation}
\widetilde{\mathcal{P}}\coloneq \{S_1\cap \cdots \cap S_d \mid S_i \in 
\widetilde{\mathcal{S}_i} \ (1\leq i \leq d) \}.
\end{equation}
By \cite[Remak 2.2.1 (3)]{BSS09} the partition $\widetilde{\mathcal{P}} \setminus \{\{0\}\}$ 
of $X\setminus \{0\}$ is a Whitney stratification. 
Recall that if $S, S^\prime \subset X$ are $\CC$-analytic their intersection 
$S\cap S^\prime \subset X$ is also $\CC$-analytic. 
This implies that the partition $\widetilde{\mathcal{P}}$ is 
a covering of $X$ by $\CC$-analytic subsets. 
Moreover it is a stratification of $X$. 
Indeed, for any member $S_1\cap \cdots \cap S_d \ (S_i \in \widetilde{\mathcal{S}_i})$ 
of $\widetilde{\mathcal{P}}$ its closure $\overline{S_1\cap \cdots \cap S_d}$ 
is a union of some members in it. 
Since Whitney's conditions (a) and (b) are trivially satisfied along 
the $0$-dimensional stratum $\{0\} \in \widetilde{\mathcal{P}}$, 
the partition $\widetilde{\mathcal{P}}$ is a Whitney stratification of $X$. 
Restricting it to $Y=\{f_1=0\}\cap \cdots \cap \{f_d=0\} \subset X$ we find that also the 
partition $\mathcal{P}$ of $Y$ is a Whitney stratification. 
We thus proved (i). 
Let us prove the remaining assertion
(ii). Let $U\subset \CC^d\setminus (C_{\{0\}}(\Delta_f)\cup \Gamma)$ 
be a neighborhood of $a$ in $\CC^d \setminus \left( C_{\{0\}}(\Delta_f) 
\cup \Gamma\right)$. Suppose that for a sequence of points $p_i=(p_{i1}, \ldots , p_{id})\in 
\CC^\ast U \subset \CC^d \setminus \left( C_{\{0\}}(\Delta_f) \cup 
\Gamma\right) \ (i=1,2,3, \ldots)$ in the open cone 
$\CC^\ast U \subset \CC^d \ (i=1,2,3,\ldots)$ 
and that of points $q_i \in f^{-1}(p_i) \ (i=1,2,3, \ldots)$ 
we have 
\begin{equation}
\begin{cases}
q_i \longrightarrow ^\exists q\in S \subset Y \quad (i\longrightarrow +\infty), \\
T_{q_i}(f^{-1}(p_i)) \longrightarrow ^\exists \mathbf{T} \ 
(\simeq \CC^{n-d}) \subset T_q X \quad (i\longrightarrow +\infty),
\end{cases}
\end{equation}
where $S$ is a stratum of the Whitney stratification of $Y$ in (i). 
Then by the condition $a\not \in C_{\{0\}}(\Delta_f)$ for any $i=1,2,3,\ldots$ 
the complex hypersurfaces $f_j^{-1}(p_{ij})\subset X \ (1\leq j\leq d)$ 
are smooth and intersect transversally. 
In particular, for the complex submanifold $f^{-1}(p_i)=f_1^{-1}(p_{i1})
\cap \cdots \cap f_d^{-1}(p_{id}) \subset X$ 
we have 
\begin{equation}
T_{q_i}(f^{-1}(p_i))=T_{q_i}(f_1^{-1}(p_{i1}))\cap \cdots \cap T_{q_i}(f_d^{-1}(p_{id})).
\end{equation}
As the Thom condition for $0$-dimensional strata in $Y$ is trivial, 
we may assume that $S\neq \{0\}$. 
Then by our construction of the Whitney stratification of $Y$, 
there exist strata $S_i\in \mathcal{S}_i\setminus \{\{0\}\} \ 
(1\leq i\leq d)$ such that $S=S_1\cap \cdots \cap S_d$. 
By \cite[Th\'eor\`eme 4.2.1]{BMM94}, for any $1\leq j\leq d$ the covector 
$df_j(q_i)\in T_{q_i}^\ast (f_j^{-1}(p_{ij}))$ 
tends to the one in $(T_{S_j}^\ast X)_q \subset T_q^\ast X$ as $i\longrightarrow +\infty$. 
Moreover by our assumption, 
the $\CC$-linear subspace 
\begin{equation}
(T_{S_1}^\ast X)_q +\cdots +(T_{S_d}^\ast X)_q \quad \subset T_q^\ast X  
\end{equation}
of $T_q^\ast X$ is the direct sum $(T_{S_1}^\ast X)_q \oplus \cdots 
\oplus (T_{S_d}^\ast X)_q \simeq (T_S^\ast X)_q$. 
This implies that the limit $\mathbf{T} \ (\simeq \CC^{n-d}) \subset T_qX$ of the tangent planes 
$T_{q_i}(f^{-1}(p_{ij})) \subset T_{q_i}X$ as $i\longrightarrow +\infty$ 
satisfies the condition $\mathbf{T}\supset T_qS$. 
This completes the proof. 
\end{proof}
\begin{example}\label{ex-str-trans}
Let $\mathbb{P}^{n-1}$ be the $(n-1)$ dimensional complex projective space and 
$\pi \colon \CC^n \setminus \{0\} \longtwoheadrightarrow \mathbb{P}^{n-1}$ 
the canonical quotient map. 
For a subvariety $Z\subset \mathbb{P}^{n-1}$ we define the cone $\widehat{Z} \subset X=\CC^n$ 
over it by $\widehat{Z}\coloneq \overline{\pi^{-1}(Z)} \subset X=\CC^n$. 
Assume that $f_i\colon X=\CC^n \longrightarrow \CC \ (1\leq i\leq d)$ 
are homogeneous polynomials and let $Z_1, \ldots , Z_d \subset \mathbb{P}^{n-1}$ be the 
projectivizations of $\{f_1=0\},\ldots , \{f_d=0\} \subset X=\CC^n$ respectively. 
For $1\leq i\leq d$ let $\mathcal{S}_i$ be a Whitney stratification of $Z_i\subset \mathbb{P}^{n-1}$ 
and assume that for any $d$-tuple $(S_1,\ldots , S_d) \ (S_i \in \mathcal{S}_i)$ of their strata 
we have the transversality $S_1\pitchfork \cdots \pitchfork S_d$. 
Then for $1\leq i\leq d$ we can naturally endow $\{f_i=0\}=\widehat{Z_i}$ 
with a Whitney stratification $\widehat{\mathcal{S}_i}$ so that the complex hypersurfaces 
$\{f_i=0\}\subset X \ (1\leq i\leq d)$ intersect stratified transversally outside the origin 
$0\in X=\CC^n$. 
It follows from Proposition \ref{prop-af} that for any $a\in \CC^d \setminus (C_{\{0\}}(\Delta_f)\cup \Gamma)$ 
the morphism $f\colon X\longrightarrow \CC^d$ satisfies the Thom condition in the direction $a$.
\end{example}
In the example below, we will present a C.I. subvariety which does not satisfy 
the condition of Definition \ref{def-thom}.  
\begin{example}\label{ex-coexample}
Let $X=\CC^5$ be the $5$-dimensional complex vector space with coordinates $(x,y,z,u,v)$ 
and $g\colon X=\CC^5 \longrightarrow \CC^2$ a holomorphic map defined by 
$(x,y,z,u,v) \longmapsto (xu, xv+yz)$. 
We denote $xu$ and $xv+yz$ by $g_1$ and $g_2$ respectively. 
It is easy to check that the subvariety $Y\coloneq g^{-1}(0)=\{g_1=0\} \cap \{g_2=0\} \subset X= \CC^n$ 
is C.I.
This example does not satisfy the Thom condition for any direction. 
Let us explain this result more precisely. 
First, by computing $2\times 2$ minors of the Jacobian matrix of $g$, 
we have $\Delta_g=\{(0,\beta) \mid \beta \in \CC\}\subset \CC^2$ and $\CC^2 \setminus C_{\{0\}}(\Delta_g) 
=\{(\alpha, \beta) \in \CC^2 \mid \alpha \neq 0\}$. 
Similarly, we have $Y_{\mathrm{sing}}=\{ x=y=z=0\} \cup \{ x=y=u=0\} \cup \{x=y=v=0\} \subset Y \subset X=\CC^5$. 
Note that by Lemma \ref{lem-hironaka}, for any Whitney stratification of $Y$ there exists a $2$-dimensional stratum 
$S\subset \{x=y=z=0\}$ so that $T_q S =\mathrm{Span}_{\CC}\{\partial_u, \partial_v\}$ 
for any $q \in S$ and the union of such strata is open dense in $\{x=y=z=0\} \subset X=\CC^5$. 
Thus for any $a\in \CC^2 \setminus C_{\{0\}}(\Delta_g)$ and any neighborhood 
$U\subset \CC^2 \setminus C_{\{0\}}(\Delta_g)$ of $a$ 
there exist a $2$-dimensional stratum $S\subset \{x=y=z=0\}$, a point $q=(0,0,0,u,v)\in S$ and 
a sequence of points $q_i\coloneq (r_i,0,0,u,v)$ such that $q_i \longrightarrow q$ 
(i.e. $r_i \longrightarrow 0$) as $i\longrightarrow +\infty$ and 
$g(q_i) \in \CC^\ast U \ (i=1,2,3,\ldots)$. 
Then the tangent plane of $q_i$ is computed as $T_{q_i}(g^{-1}(g(q_i)))=\ker dg_{q_i} 
=\mathrm{Span}_{\CC}\{ \partial_y, \partial_z, r_i \partial_x -u\partial_u -v\partial_v\}$.
Since each tangent plane of $q_i$ is dependent on only $r_i$, 
we obtain 
\begin{equation}
\lim_{i\to +\infty} T_{q_i}(g^{-1}(g(q_i))) =\mathrm{Span}_{\CC}\{ \partial_y, 
\partial_z, u\partial_u +v\partial_v \} \not \supset T_q S= \mathrm{Span}_{\CC}\{\partial_u, \partial_v\}.
\end{equation}
This means that the Thom condition does not hold true in any direction. 
\end{example}

\begin{remark} 
This subvariety is a example which does not satisfy the stratified transversality 
in Definition \ref{def-str-transversality}.
In order to see this, let us consider the subset $\{(0,0,0,u,0) \in \CC^5 \mid u\neq 0\} \subset Y$, 
which is a intersection of the regular part 
$\{g_1=0\} \setminus \{(0,y,z,0,v) \in \CC^5\} \subset \{g_1=0\}$ of $\{g_1=0\}$ 
and the singular part $\{(0,0,0,u,0) \in \CC^5 \} \subset \{g_2=0\}$ of $\{g_2=0\}$. 
Remark that by Lemma \ref{lem-hironaka}, for any Whitney stratification of $\{g_1=0\}$ (resp. $\{g_2=0\}$) 
there exists a $4$-dimensional stratum $S^{(1)} \subset \{g_1=0\}$ (resp. 
$1$-dimensional stratum $S^{(2)}\subset \{g_2=0\}$) 
contained in $\{(0,y,z,u,v)\in \CC^5 \mid u\neq 0\}$ (resp. singular part of $\{g_2=0\}$). 
Then for any point $q\in S^{(1)}\cap S^{(2)} \subset \{(0,0,0,u,0) \in \CC^5 \mid u\neq 0\}$ 
we have $T_q S^{(1)}\simeq \mathrm{Span}_{\CC} \{\partial_y, \partial_z, \partial_u, \partial_v\}$ 
and $T_q S^{(2)} \simeq \mathrm{Span}_{\CC} \{\partial_u\}$ 
and thus obtain 
\begin{equation}
T_qS^{(1)} + T_qS^{(2)} \neq T_q X.
\end{equation}
Note that the condition of Definition \ref{def-transversality} is equivalent to the 
usual definition of transversality in the case of $d=2$.
\end{remark}

\begin{proposition}\label{prop-milfib}
Assume that $f=(f_1,\ldots , f_d)\colon X \longrightarrow \CC_s^d$ 
satisfies the Thom condition in a direction $a\in \CC^d \setminus \mathrm{C}_{\{0\}}(\Delta_f)$. 
Then for any point $q\in Y\subset X=\CC^n$ there exist isomorphisms 
\begin{equation}
H^j \mathrm{Sp}_{X\times \{0\} \mid X\times \CC^d} ({i_f}_\ast \CC_X)_{(q,a)} \simeq 
H^j(M_{Y,q}; \CC) \quad (j\in \ZZ).
\end{equation}
\end{proposition}

\begin{proof}
We may assume that $q=0\in Y\subset X=\CC^n$. 
Let $\widetilde{(X\times \CC^d)_{X\times \{0\}}} \simeq X\times \CC_s^d \times \CC_t$ 
be the deformation to the normal cone of $X\times \{0\} \subset X\times \CC^d$ 
and consider the morphisms 
\begin{equation}
\begin{cases}
p\colon \widetilde{(X\times \CC^d)_{X\times \{0\}}} 
\longrightarrow X\times \CC^d \quad &((x,s,t) \longmapsto (t,ts)), \\
t\colon \widetilde{(X\times \CC^d)_{X\times \{0\}}} \longrightarrow \CC \quad & ((x,s,t)\longmapsto t) 
\end{cases}
\end{equation}
associated to it such that $t^{-1}(0) \simeq \mathrm{C}_{X\times \{0\}}(X\times \CC^d) \simeq X\times \CC^d$. 
Then there exist isomorphisms 
\begin{align}
\mathrm{Sp}_{X\times \{0\} \mid X\times \CC^d} ({i_f}_\ast F) &\simeq \psi_t(p^{-1}{i_f}_\ast \CC_X) \\
& \simeq \psi_t(\CC_{\{ts=f(x)\}}).
\end{align}
For $\e, \delta>0$ let $B(0;\e)\subset X=\CC^n$ (resp. $B(a;\delta)\subset \CC_s^d$) 
be the open ball centered at $0\in X=\CC^n$ (resp. $a\in \CC^d$) 
with radius $\e>0$ (resp. $\delta>0$). Then by Proposition \ref{prop-cont} 
there exist $0<\e_1 \ll 1$ and $\lambda >0$ such that we have isomorphisms 
\begin{align}
\psi_t(\CC_{\{ts=f(x)\}})_{(0,a)} &\simeq \rsect (B(0;\e)\times B(a;\lambda \e)\times \{\tau\}; \CC_{\{ts=f(x)\}}) \\
&\simeq \rsect(B(0;\e)\times B(a;\lambda \e)); \CC_{\{\tau s=f(x)\}})
\end{align}
for any $0<\e\leq \e_1$ and $\tau \in \CC^\ast$ such that $0<\abs{\tau} \ll 1$.
Let $Y=\bigsqcup_{\alpha \in A}Y_{\alpha}$ be a Whitney stratification of $Y\subset X=\CC^n$ 
(satisfying the condition in Definition \ref{def-thom}). 
Then by replacing $0<\e_1 \ll 1$ if necessary we may assume also that $\partial B(0;\e)\subset X=\CC^n$ 
intersects $Y_{\alpha}$ transversally for any $0<\e \leq \e_1$ and $\alpha \in A$. 
Together with our assumption, we thus can see that $f^{-1}(\tau s) \subset X=\CC^n$ 
intersects $\partial B(0;\e) \subset X=\CC^n$ transversally for any 
$0<\e \leq \e_1$, $s\in \overline{B(a;\lambda \e)}$ and $\tau >0$ such that $0<\tau \ll 1$. 
Now we define a real analytic function $\varphi \colon X\times \CC^d \longrightarrow \RR$ 
on $X\times \CC^d$ by 
\begin{equation}
\varphi(x,s) \coloneq \sqrt{\abs{s_1-a_1}^2+\cdots +\abs{s_d-a_d}^2} \quad ((x,s)\in X\times \CC^d)
\end{equation}
and set 
\begin{equation}
\mathcal{K} \coloneq \rsect_{B(0;\e_1)\times \CC^d}(\CC_{\{\tau s=f(x)\}}) \quad \in \BDC_{\RR-c}(X\times \CC^d)
\end{equation}
for some fixed $\tau>0$ such that $0<\tau \ll 1$. 
Then we see that the subset 
\begin{equation}
\{(x,s; d\varphi(x,s)) \mid (x,s) \in X\times \CC^d, \quad \varphi(x,s)\leq 
\lambda \e_1\} \quad \subset T^\ast(X\times \CC^d)
\end{equation}
of $T^\ast(X\times \CC^d)$ meets $\msupp(\mathcal{K})$ only in the zero-section 
of $T^\ast(X\times \CC^d)$.
By the non-characteristic deformation lemma \cite[Proposition 2.7.2]{KS90} 
we thus obtain isomorphism 
\begin{align}
&\rsect(B(0;\e_1)\times B(a;\lambda \e_1);\CC_{\{\tau s=f(x)\}}) \\
&\simeq \rsect(\varphi^{-1}([0,\lambda \e_1]);\mathcal{K}) \\
&\simto \rsect(\varphi^{-1}(0);\mathcal{K})=\rsect(X\times \{a\};\mathcal{K}) \\
&\simeq \rsect(B(0;\e) ;\CC_{\{\tau a=f(x)\}}).
\end{align}
From this the assertion immediately follows.
\end{proof}

\begin{proposition}\label{prop-versp}
Assume that $f=(f_1,\ldots , f_d)\colon X\longrightarrow \CC^d$ 
satisfies the Thom condition in a direction $a\in \CC^d \setminus \mathrm{C}_{\{0\}}(\Delta_f)$.
Then there exists a neighborhood $U\subset \CC^d \setminus \mathrm{C}_{\{0\}}(\Delta_f)$ 
of $a$ in $\CC^d \setminus \mathrm{C}_{\{0\}}(\Delta_f)$ and a constructible sheaf 
$F\in \BDC_c(X)$ on $X$ such that for the projection $\alpha_U \colon X\times U \longrightarrow X$ 
we have an isomorphism 
\begin{equation}
\mathrm{Sp}_{X\times \{0\} \mid X\times \CC^d}({i_f}_\ast \CC_X)|_{X\times U} \simeq \alpha_U^{-1}F.
\end{equation}
\end{proposition}

\begin{proof}
Let $U\subset \CC^d \setminus \mathrm{C}_{\{0\}}(\Delta_f)$ be a neighborhood 
of $a$ satisfying the condition in Definition \ref{def-thom} and 
set $\mathcal{F} \coloneq \mathrm{Sp}_{X\times \{0\} \mid X\times \CC^d}
({i_f}_\ast \CC_X) \in \BDC_c(X\times \CC^d)$. 
For $\delta>0$ and a point $p\in \CC^d$ let $B(p;\delta)\subset \CC^d$ 
be the open ball centered at $p$ with radius $\delta>0$. 
Then by \cite[Proposition 5.4.5 (ii)]{KS90} it suffices to show that for any 
$q\in Y\subset X=\CC^n$ and $p\in U$ there exists $0<\delta_0 \ll 1$ such that 
$\overline{B(p;\delta_0)} \subset U$ and the canonical morphisms 
\begin{equation}
\rsect(\{q\}\times B(p;\delta_0);\mathcal{F})\longrightarrow 
\mathcal{F}_{(q,p^\prime)} \quad (p^\prime \in B(p;\delta_0))
\end{equation}
are isomorphisms. 
For this purpose, we shall use the results in Section \ref{sec-pre}.
Let $\mathcal{S}$ be a complex analytic stratification of $X\times \CC^d$ 
adapted to $\mathcal{F} \in \BDC_c(X\times \CC^d)$ and for $\e>0$ let $B(q;\e) \subset X=\CC^n$ 
be the open ball centered at $q\in Y\subset X=\CC^n$ in $X=\CC^n$ with radius $\e>0$. 
Then by the arguments in Section \ref{sec-pre} 
there exist $0<\e_0 \ll 1$ and a 1-dimensional subanalytic subset $E\subset \RR^2$ 
contained in $\{(\e,\delta)\in\RR^2 \mid 0<\e \leq \e_0, \delta>0 \} \subset\RR^2$ 
such that for any $0<\e \leq \e_0$ and $\delta>0$ satisufying the condition 
$(\e, \delta)\not \in E$ the three submanifolds 
\begin{equation}
\partial B(q;\e) \times \partial B(p;\delta), \quad 
\partial B(q;\e) \times B(p;\delta), \quad 
B(q;\e) \times \partial B(p;\delta)
\end{equation}
of $X\times \CC^d$ intersect any stratum $S\in \mathcal{S}$ in $\mathcal{S}$ transversally.
Note that the projection $\RR^2 \longrightarrow \RR$ ($(\e, \delta) \longmapsto \delta$) 
is proper on that closure $\overline{E}$ of $E$ in $\RR^2$. 
We thus can apply the result in Goresky-MacPherson \cite[page 43]{GM88} to show 
that there exist $0<\delta_0 \ll 1$ and a subanalytic funcntion $\phi 
\colon (0,\delta_0) \longrightarrow \RR_{>0}$ 
such that $\overline{B(p;\delta_0)} \subset U$ and 
\begin{equation}
\{ (\e,\delta)\in \RR^2 \mid 0<\delta \leq \delta_0, 0<\e <\phi(\delta) \} \cap E =\varnothing.
\end{equation}
Then by the non-characteristic deformation lemma \cite[Proposition 2.7.2]{KS90} 
we obtain an isomorphism 
\begin{equation}
\rsect(B(q;\e)\times B(p;\delta_0);\mathcal{F}) \simto 
\rsect(\{q\}\times B(p;\delta_0);\mathcal{F})
\end{equation}
for any $0<\e <\phi(\delta_0)$. 
Let $Y=\bigsqcup_{\alpha \in A}Y_{\alpha}$ be a complex analytic Whitney stratification 
of $Y\subset X=\CC^n$ and take $0<\e_1 \leq \min{\{\e_0, \phi(\delta_0)\}}$ 
such that for any $0<\e \leq \e_1$ and $\alpha \in A$ the submanifold 
$\partial B(q;\e) \subset X=\CC^n$ intersects $Y_{\alpha} \subset X=\CC^n$ transversally. 
Then as in the proof of Proposition \ref{prop-milfib}, by using our assumption and the 
non-characteristic deformation lemma \cite[Proposition 2.7.2]{KS90} 
we can show that for any $p^\prime \in B(p;\delta_0)$ there exist isomorphisms 
\begin{equation}
\rsect(B(q;\e)\times B(p;\delta_0) ;\mathcal{F})\simto 
\rsect(B(q;\e)\times \{p^\prime\};\mathcal{F}) \quad (0<\e\leq \e_1).
\end{equation}
Moreover, for $0<\e \ll \e_1$ we have an isomorphism 
\begin{equation}
\rsect(B(q;\e)\times \{p^\prime\};\mathcal{F}) \simto \mathcal{F}_{(q,p^\prime)}
\end{equation}
and hence obtain the assertion. This completes the proof.
\end{proof}
Now recall that the Verdier specialization functor 
\begin{equation}
\mathrm{Sp}_{X\times \{0\} \mid X\times \CC^d}(\cdot)\colon 
\BDC_c(X\times \CC^d) \longrightarrow \BDC_c(X\times \CC^d)
\end{equation}
preserves the perversity and consider the perverse sheaf 
\begin{equation}
\mathcal{G} \coloneq \mathrm{Sp}_{X\times \{0\} \mid X\times 
\CC^d}({i_f}_\ast \CC_X[n]) \quad \in \Perv(\CC_{X\times \CC^d})
\end{equation}
on $\mathrm{C}_{X\times \{0\}}(X\times \CC^d) \simeq X\times \CC^d$. 
Then, in the situation of Proposition \ref{prop-versp}, there 
exists a perverse sheaf $G\in \Perv(\CC_X)$ 
on $X$ such that for any point $p\in U \subset \CC^d \setminus 
\mathrm{C}_{\{0\}}(\Delta_f)$ 
we have an isomorphism 
\begin{equation}
\mathcal{G}|_{X\times \{p\}}\simeq G[d].
\end{equation}
In the situation of Proposition \ref{prop-milfib}, let $i_a\colon 
X\times \{a\} \longhookrightarrow X\times \CC^d$ 
and $\widetilde{i_a} \colon X\times \{a\} \times \CC \longhookrightarrow 
X\times \CC^d \times \CC =
\widetilde{(X\times \CC^d)_{X\times \{0\}}}$ 
be the inclusion maps and for the projection 
$t\colon X\times \CC^d \times \CC \longrightarrow \CC$ identify 
$t^{-1}(0)=X\times \CC^d \times \{0\}$ 
with $X\times \CC^d$.
Then by the proof of Proposition \ref{prop-milfib} and 
\cite[Theorem 2.5]{Tak25} the natural morphism 
\begin{align}
G&\simeq \mathcal{G}|_{X\times \{a\}}[-d] \simeq i_a^{-1} 
\psi_t(p^{-1}{i_f}_\ast \CC_X[n-d]) \\
& \longrightarrow i_a^{-1}\psi_t(\widetilde{i_a}_\ast 
\widetilde{i_a}^{-1} p^{-1}{i_f}_\ast \CC_X[n-d]) \\
& \simeq i_a^{-1}{i_a}_\ast \psi_t(\CC_{\{ta=f(x)\}}[n-d]) 
\simeq \psi_t(\CC_{\{ta=f(x)\}}[n-d])
\end{align}
is an isomorphism. 

\section{Euler obstructions of complex hypersurfaces}\label{sec-hypersurface}
In this section, based on the results in \cite{FKT26}, 
we prove a formula for the Euler obstructions of complex hypersurfaces 
(with possibly non-isolated singular points) 
in $X=\CC^n$. 
Let $f\colon X \longrightarrow \CC$ be a holomorphic function on $X=\CC^n$ 
such that $f(0)=0$ and set $Y\coloneq \{f=0\} \subset X=\CC^n$. 
Assume that $f$ is reduced i.e. it generates the defining ideal 
$\mathcal{I}_Y \subset \mathcal{O}_X$ of $Y$. 
Let $Y=\bigsqcup_{\alpha \in A}Y_{\alpha}$ be a Whitney stratification of $Y=\{f=0\} \subset X$ 
consisting of connected strata $Y_\alpha \ (\alpha \in A)$. 
Then it is well-known that the Euler obstruction 
$\mathrm{Eu}_Y\colon Y\longrightarrow \ZZ$ of $Y$ is constant on each stratum $Y_\alpha$ 
(see e.g. \cite[Theorem 8.1.1 (4)]{BSS09}). 
Our objective here is to describe its value $\mathrm{Eu}_Y(0) \in \ZZ$ 
at the origin $0\in Y \subset X=\CC^n$. 
We denote $\alpha \in A$ such that $0\in Y_\alpha$ by $\alpha_0$. 
Then we may assume that $\dim Y_{\alpha_0}=0$ i.e. $Y_{\alpha_0}=\{0\}$. 
Indeed, suppose that to the contrary we have $\dim Y_{\alpha_0} >0$. 
In such a case, we can take a linear subspace $W \ (\simeq \CC^{n-\dim Y_{\alpha_0}}) \subset X=\CC^n$ 
of dimension $n-\dim Y_{\alpha_0}$ which intersects $Y_{\alpha_0}$ 
transversally at the origin $0\in Y_{\alpha_0}$. 
Namely $W$ is a normal slice of $Y_{\alpha_0}$ at the origin. 
Then by the Whitney conditions of the stratification $Y=\bigsqcup_{\alpha \in A} Y_\alpha$, 
we have $\mathrm{Eu}_Y(0)=\mathrm{Eu}_{Y\cap W}(0)$ and hence can 
replace $Y\subset X=\CC^n$ by $Y\cap W \subset W\simeq \CC^{n-\dim Y_{\alpha_0}}$. 
Let $i_f \colon X \longhookrightarrow X\times \CC_t \ (x\longmapsto (x,f(x)))$ be the graph 
embedding of $f\colon X\longrightarrow \CC$ and $t\colon X\times \CC \longrightarrow \CC$ 
the projection.
Then we define a perverse sheaf $G$ on $X=\CC^n$ by 
\begin{equation}
G\coloneq \psi_t({i_f}_\ast (\CC_X[n]))[-1] \quad \in \Perv(\CC_X).
\end{equation}
Note that by \cite[Theorem 2.5]{Tak25} for the inclusion map $\iota \colon Y \longhookrightarrow X$ 
there exists an isomorphism $G\simeq \iota_\ast \psi_f(\CC_X[n-1])$. 
By the Whitney stratification $Y=\bigsqcup_{\alpha \in A}Y_\alpha$ of $Y$ 
we naturally endow $X=\CC^n$ with the one 
\begin{equation}
X=\left( \bigsqcup_{\alpha \in A}Y_\alpha\right) \sqcup (X\setminus Y).
\end{equation}
Then by applying Proposition \ref{prop-adapt} to the perverse sheaf 
${i_f}_\ast(\CC_X [n]) \in \Perv(\CC_{X\times \CC})$ on $X\times \CC$ 
we see that $G\in \Perv(\CC_X)$ is adapted to this Whitney stratification of $X=\CC^n$. 
By Lemma \ref{lem-perversity} this implies that we have 
\begin{equation}
\msupp (G) \subset \bigsqcup_{\alpha \in A} T_{Y_\alpha}^\ast X.
\end{equation}
To begin with, let us give a formula for the multiplicity $m_0\geq 0$ of 
$[T_{Y_{\alpha_0}}^\ast X]=[T_{\{0\}}^\ast X]$ in the characteristic cycle $\CCyc(G)$ 
of the perverse sheaf $G$. 
First, note that for a generic linear form $\ell \colon X=\CC^n \longrightarrow \CC$ 
on $X=\CC^n$ we have the condition 
\begin{equation}
d\ell(0) =(0;dx_1) \in \Omega \setminus \bigcup_{\alpha: \alpha 
\neq \alpha_0}\overline{T_{Y_\alpha}^\ast X},
\end{equation}
where $\Omega \subset T_{\{0\}}^\ast X$ is an open dense subset 
satisfying some conditions (see Appendix \ref{ap-A} for the 
details). We fix such $\ell$ once and for all in this section and define a hyperplane $H\subset X$ 
of $X=\CC^n$ by $H\coloneq \ell^{-1}(0)$. Let $M_{Y,0} \subset X\setminus Y$ 
(resp. $M_{Y\cap H, 0}\subset H\setminus (Y\cap H)$) be the Milnor fiber of the 
complex hypersurface $Y\subset X=\CC^n$ (resp. $Y\cap H \subset H \simeq \CC^{n-1}$) 
at the origin $0\in Y$ (resp. $0\in Y\cap H$). 
Then by the proof of \cite[Theorem 5.5]{FKT26} we obtain 
\begin{equation}\label{eq-mult}
m_0=(-1)^{n-1} \left\{\chi(M_{Y,0}) -\chi(M_{Y\cap H, 0})\right\}.
\end{equation}
Also for $\alpha \in A$ such that $\alpha \neq \alpha_0$, $0\in \overline{Y_\alpha}$ 
and $\dim Y_\alpha < \dim Y$ we can describe the multiplicity of $[T_{Y_\alpha}^\ast X]$ 
in $\CCyc(G)$ as follows. 
Note that for such $\alpha \in A$ we have $\dim Y_\alpha \geq 1$. 
Let $W_\alpha \ (\simeq \CC^{n-\dim Y_\alpha}) \subset X=\CC^n$ be an affine subspace 
of dimension $n-\dim Y_\alpha$ which intersects $Y_\alpha$ transversally at a point $p_\alpha \in Y_\alpha$. 
We call it a normal slice of $Y_\alpha$ at $p_\alpha \in Y_\alpha$. 
Then there exists a neighborhood $U_\alpha \subset W_\alpha$ of $p_\alpha \in W_\alpha \cap Y_\alpha$ 
in $W_\alpha \simeq \CC^{n-\dim Y_\alpha}$ such that 
$\{Y_{\alpha^\prime} \cap U_\alpha \mid \alpha^\prime \in A\}$ is a Whitney stratification 
of $Y\cap U_\alpha \subset W_\alpha$ and $G_\alpha \coloneq G|_{U_\alpha}[-\dim Y_\alpha] \in \BDC_c(U_\alpha)$ 
is a perverse sheaf on $U_\alpha \subset W_\alpha \simeq \CC^{n-\dim Y_\alpha}$. 
Moreover, as in the proof of Proposition \ref{prop-milfib} we can easily show that 
there exists an isomorphism 
\begin{equation}
G_\alpha \simeq \psi_t({i_f}_\ast (\CC_X)|_{U_\alpha \times \CC}[n-1-\dim Y_\alpha]). 
\end{equation}
Let $m_\alpha \geq 0$ be the multiplicity of $[T_{\{p_\alpha\}}^\ast U_\alpha]$ 
in the characteristic cycle $\CCyc({G_\alpha})$ of the perverse sheaf $G_\alpha \in \Perv(\CC_{U_\alpha})$ 
on $U_\alpha$. 
Then the multiplicity of $[T_{Y_\alpha}^\ast X]$ in $\CCyc(G)$ 
is equal to $(-1)^{\dim {Y_\alpha}} m_\alpha \in \ZZ$. 
Indeed, for a point 
\begin{equation}
q_{\alpha} \in T_{\{p_{\alpha} \}}^\ast U_{\alpha} \setminus 
\bigcup_{\alpha^{\prime} \not= \alpha, \ Y_{\alpha} \subset 
\overline{Y_{\alpha^{\prime}}}} 
\overline{T_{Y_{\alpha^{\prime}} \cap U_{\alpha} }^\ast U_{\alpha}}
\end{equation}
we take its sufficiently small neighborhood $\Omega_{q_{\alpha}}$ in $T_{\{p_{\alpha}
\}}^\ast U_{\alpha}$. Then by the perversity of $G_{\alpha}$ there exists an isomorphism 
$G_{\alpha} \simeq \CC_{\{p_{\alpha} \}}^{\oplus m_{\alpha}}$ 
in the localized category $\BDC(U_{\alpha};\Omega_{q_{\alpha}})$ (for the 
definition, see \cite[Definition 6.1.1]{KS90}). Let 
\begin{equation}
T^\ast U_{\alpha} \underset{\rho}{\longtwoheadleftarrow} 
 U_{\alpha} \times_X T^\ast X \underset{\varpi}{\longhookrightarrow} T^\ast X 
\end{equation}
be the natural morphisms associated to the inclucion map $U_{\alpha} \longhookrightarrow X$. 
We take a suffciently small neighborhood $\widetilde{\Omega_{q_{\alpha}}}$ 
of $\varpi \rho^{-1} \Omega_{q_{\alpha}} \subset T^\ast X$ 
in $T^\ast X$. 
Then by the perversity of $G$ there exists a non-negative integer 
$m\geq 0$ such that we have 
an isomorphism $G\simeq \CC_{Y_{\alpha}}^{\oplus m} [\dim Y_{\alpha}]$ 
in the localized category 
$\BDC(X;\widetilde{\Omega_{q_{\alpha}}})$. 
This implies that the multiplicity of $[T_{Y_{\alpha}}^\ast X]$ in $\CCyc(G)$ 
is equal to $(-1)^{\dim Y_{\alpha}}m$. 
Moreover, by \cite[Proposition 5.4.13 (i)]{KS90} there exists an isomorphism 
\begin{equation}
G_{\alpha} =G|_{U_{\alpha}}[-\dim Y_{\alpha}] \simeq \CC_{\{p_{\alpha} \}}^{\oplus m}
\end{equation}
in $\BDC(U_{\alpha}; \Omega_{q_{\alpha}})$. We thus obtain 
the desired equality $m=m_{\alpha}$. 
Let $H_\alpha \subset W_\alpha$ be a generic hyperplane in $W_\alpha \simeq \CC^{n-\dim Y_\alpha}$ 
passing through the point $p_\alpha \in W_\alpha \cap Y_\alpha$. 
Let $M_{Y\cap U_\alpha, p_\alpha} \subset U_\alpha \setminus (Y\cap U_\alpha)$ 
(resp. $M_{Y\cap H_\alpha \cap U_\alpha, p_\alpha} \subset (H_\alpha 
\cap U_\alpha) \setminus (Y\cap H_\alpha \cap U_\alpha)$) 
be the Milnor fiber of the complex hypersurface $Y\cap U_\alpha \subset U_\alpha$ 
(resp. $Y\cap H_\alpha \cap U_\alpha \subset H_\alpha \cap U_\alpha$) at the point 
$p_\alpha \in Y\cap U_\alpha$ (resp. $p_\alpha \in Y\cap H_\alpha \cap U_\alpha$). 
Then by the proof of \cite[Theorem 5.5]{FKT26} we obtain the following result. 
\begin{theorem}\label{thm-mult-hyp}
For any $\alpha \in A$ such that $\alpha \neq \alpha_0, 0\in \overline{Y_\alpha}$ and 
$\dim Y_\alpha <\dim Y$ we have 
\begin{equation}
m_\alpha =(-1)^{\dim Y -\dim Y_\alpha} \left\{\chi(M_{Y\cap U_\alpha, p_\alpha})-
\chi(M_{Y\cap H_\alpha \cap U_\alpha, p_\alpha})\right\}.
\end{equation}
\end{theorem}
From this, we obtain the following main theorem in this section. 
\begin{theorem}\label{thm-main}
We have an equality 
\begin{align}
\mathrm{Eu}_Y(0)=&\chi(M_{Y\cap H, 0}) \\
&-\sum_{\substack{\alpha \neq \alpha_0, 0\in \overline{Y_\alpha} \\ \dim 
Y_\alpha <\dim Y}}\left\{\chi(M_{Y\cap U_\alpha, p_\alpha})-\chi(M_{Y\cap 
H_\alpha \cap U_\alpha, p_\alpha})\right\} \cdot \mathrm{Eu}_{\overline{Y_\alpha}}(0).
\end{align}
\end{theorem}
\begin{proof}
By the Riemann-Hilbert correspondence, there exists a regular 
holonomic $\SD_X$-module $\SM$ on $X=\CC^n$ such that 
\begin{equation}
\Sol_X(\SM)[n] \simeq G=\psi_t({i_f}_\ast (\CC_X[n-1])).
\end{equation}
Then by $n-1=\dim Y$ and \cite[Theorem 2.6]{Tak25} we obtain an equality 
\begin{equation}\label{eq-chisol}
\chi(\Sol_X(\SM))(0) =(-1)^{\dim X -\dim Y} \chi (M_{Y,0}).
\end{equation}
For $\alpha \in A$ let  $\mult_{T_{Y_\alpha}^\ast X}(\SM)\geq 0$ be the multiplicity 
of $\SM$ along the Lagrangian submanifold $T_{Y_\alpha}^\ast X \subset T^\ast X$. 
Then by Proposition \ref{prop-mult} 
for any $\alpha \in A$ such that $\alpha \neq \alpha_0, 0\in \overline{Y_\alpha}$ 
and $\dim Y_\alpha < \dim Y=n-1$ we have $m_\alpha=\mult_{T_{Y_\alpha}^\ast X}(\SM)$. 
Moreover we have $m_0=\mult_{T_{\{0\}}^\ast X}(\SM)$. 
In this situation, Kashiwara's index theorem in \cite[Chapter 6]{Kas83a} asserts that 
\begin{equation}\label{eq-chi-hyp}
\begin{aligned}
&\chi(\Sol_X(\SM))(0)=(-1)^{\dim X -\dim Y} \mathrm{Eu}_Y(0) \\
&+ \sum_{\alpha \neq \alpha_0, 0\in \overline{Y_\alpha}, \dim Y_\alpha <\dim Y} 
(-1)^{\dim X - \dim Y_\alpha} m_\alpha \cdot \mathrm{Eu}_{\overline{Y_\alpha}}(0) 
+ (-1)^n m_0 \cdot \mathbf{1}_{\{0\}}.
\end{aligned}
\end{equation}
Then by \eqref{eq-mult}, Theorem \ref{thm-mult-hyp} and \eqref{eq-chi-hyp} we obtain the assertion. 
\end{proof}
In the example below, we will show that 
our main result Theorem \ref{thm-main} allows us to 
obtain a well-known value of the local Euler obstruction 
at the origin. 
\begin{example}
Let $X=\CC^3$ be the 3-dimensional complex vector space 
with coordinates $(x,y,z)$. 
For a holomorphic function $f\colon X=\CC^3 \longrightarrow \CC$ 
set $Y\coloneq f^{-1}(0)\subset X$. 
In what follows, we use the same notations as in Theorem \ref{thm-main}.  
\begin{enumerate}
\item [(i)]
Let us consider the case of $f(x,y,z)=y^2-x^2z$. 
In this situation, we can take a stratification $Y=\bigsqcup_{i=0}^2 Y_i$ as follows: 
\begin{equation}
\begin{cases}
Y_0\coloneq \{0\} \\
Y_1\coloneq \{(x,y,z)\in \CC^3 \mid x=y=0\}\setminus Y_0 \\
Y_2\coloneq Y\setminus (Y_0 \sqcup Y_1).
\end{cases}
\end{equation}
It is easy to check that this stratification satisfies the Whitney condition. 
Note that since $p_\alpha \in Y\cap H_\alpha \cap U_\alpha \subset Y\cap U_\alpha$ 
is an isolated singular point, by the classical result in \cite{Kas83a}, 
we have $\chi(M_{Y\cap H_\alpha \cap U_\alpha,p_\alpha})=\mathrm{Eu}_{Y\cap U_\alpha}(p_\alpha)$.
Thus by Theorem \ref{thm-main}, 
it suffices to calculate the Euler characteristic of 
a generic hyperplane section of $Y$ at the origin $0\in X$,  
the Milnor number of the normal slice of $Y_1$ and its Euler obstruction of a 
reference point. 
First, let us define a complex hyperplane $H$ by 
\begin{equation}
H\coloneq \{ (x,y,z)\in \CC^3 \mid x+z=0\},
\end{equation}
which passes through the origin $0\in X=\CC^3$ and 
satisfies the condition 
\begin{equation}
T_H^\ast X \cap (\overline{T_{Y_1}^\ast X} \cup \overline{T_{Y_2}^\ast X}) \subset T_X^\ast X.
\end{equation}
For such $H\subset X=\CC^3$ the complex hypersurface $Y\cap H \subset H$ 
in it has an isolated singular point at the origin and 
is defined by the polynomial $y^2+x^3$. 
Since this polynomial is of Brieskorn-Pham type, 
its Milnor number is computed as $(2-1)(3-1)=2$. 
Thus we have $\chi(M_{Y\cap H,0})=1-2=-1$. 
It follows that for $a \in \CC^\ast$ we take a normal slice 
$S\coloneq \{ (x,y,z)\in \CC^3 \mid z=a \neq 0\}$ of $Y_1\subset X=\CC^3$. 
Then $Y\cap S \subset S$ is defined by the polynomial 
$y^2-ax^2$ and its Milnor number at  
$(0,0,a)\in Y \cap S \subset S$ in $S$ is 1. 
Moreover, by the result of \cite{Kas73}, for 
$1$-dimensional varieties the value of Euler obstructions is equal to the multiplicity of it. Thus in this case the value of the Euler obstruction 
of a reference point is $2$.
Therefore, by Theorem \ref{thm-main} we obtain 
\begin{equation}
\mathrm{Eu}_Y(0) =-1-\{(1-1)-2
\}\cdot 1=1.
\end{equation}
\item[(ii)]
Let us consider the case of $f(x,y,z)=x^n+y^pz^q$, where $n,p,q\geq 2$ and 
the greatest common divisor of $n,p \ \text{and} \ q$ is $1$. Then we can take 
a Whitney stratification $Y=\bigsqcup_{i=0}^3Y_i$ as follows: 
\begin{equation}
\begin{cases}
Y_0\coloneq \{0\} \\ 
Y_1\coloneq \{(x,y,z)\in \CC^3 \mid x=y=0\}\setminus Y_0 \\
Y_2\coloneq \{(x,y,z)\in \CC^3 \mid x=z=0\}\setminus Y_0 \\
Y_3\coloneq Y\setminus (Y_0 \sqcup Y_1 \sqcup Y_2).
\end{cases}
\end{equation}
The Milnor numbers and multiplicities of the respective normal slices can be computed in the same way as in the case (i). 
For example, take a complex hyperplane $H$ by
\begin{equation}
H\coloneq \{ (x,y,z)\in \CC^3 \mid y+z=0\},
\end{equation}
which passes through the origin $0\in X=\CC^3$ and 
satisfies the condition 
\begin{equation}
T_H^\ast X \cap (\overline{T_{Y_1}^\ast X} \cup \overline{T_{Y_2}^\ast X} \cup \overline{T_{Y_3}^\ast X}) \subset T_X^\ast X.
\end{equation}
Then as in the case (i) we have $\chi(M_{Y\cap H, 0})=1-(n-1)(p+q-1)$. 
Moreover, for $a\in \CC^\ast$ we can take a normal slice 
$S\coloneq \{(x,y,z)\in \CC^3 \mid z=a\neq 0 \}\subset X=\CC^3$ of 
$Y_1\subset X=\CC^3$ at $(0,0,a) \in Y_1\cap S\subset S$. 
Then it follows that the Milnor number of $Y\cap S$ at $(0,0,a) \in Y\cap S$ 
is $(n-1)(p-1)$ and its multiplicity at $(0,0,a)\in Y\cap S$ is $\min\{n,p\}$.
The other term can be computed similarly. 
Therefore, by Theorem \ref{thm-main} we obtain 
\begin{equation}
\begin{aligned}
\mathrm{Eu}_Y(0)=&1-(n-1)(p+q-1)\\
&-\{1-(n-1)(p-1)-\min\{n,p\}\}\cdot 1 \\
&-\{1-(n-1)(q-1)-\min\{n,q\}\}\cdot 1 \\
&=\min\{n,p\} +\min\{n,q\}-n.
\end{aligned}
\end{equation}
\end{enumerate}
\end{example}

\section{Euler obstructions of complete intersection varieties}\label{Eu-CI}
In this section,  
we prove a formula for the Euler obstructions of complete intersection subvarieties
(with possibly non-isolated singular points) in $X=\CC^n$. 
We inherit the situations and the notations in Section \ref{sec-verdier}. 
Throughtout this section, we always assume that 
$f=(f_1,\ldots ,f_d)\colon X\longrightarrow \CC_s^d$ satisfies 
the Thom condition in a direction $a\in \CC^d \setminus C_{\{0\}}(\Delta_f)$. 
First, we shall study the characteristic cycle $\CCyc(G)$ of the perverse sheaf 
$G\simeq \psi_t(\CC_{\{ta=f(x)\}}[n-d])$ on $X\simeq X\times \{a\}$. 
For this purpose, let $\mathcal{S}$ be a 
Whitney stratification of the analytic subset 
\begin{equation}
\{ (x,t)\in X\times \CC \mid ta =f(x)\}=\supp(\CC_{\{ta=f(x)\}})
\end{equation}
of $X\times \CC$ such that $\supp(\CC_{\{ta=f(x)\}}) \cap (X\times \{0\}) =Y\times \{0\}$ 
(resp. $\supp(\CC_{\{ta=f(x)\}}) \setminus (X\times \{0\})$) is a union of some strata 
(resp. a stratum) in it and set $\mathcal{S}_0 \coloneq \{S\in \mathcal{S} \mid S\subset Y\times \{0\}\}$. 
Then by Lemma \ref{lem-perversity} and Proposition \ref{prop-adapt} we have 
\begin{equation}
\msupp(G) \subset \bigsqcup_{S\in \mathcal{S}_0}T_S^\ast X.
\end{equation}
Subdividing $\mathcal{S}_0$ if necessary, we may assume that the Whitney stratification 
$\mathcal{S}_0$ of $Y\times \{0\}\simeq Y$ satisfies the condition in Definition \ref{def-thom}.
First, for a zero dimensional stratum $S\in \mathcal{S}_0$ in $\mathcal{S}_0$ let us 
describe the multiplicity of $[T_S^\ast X]$ in the characteristic cycle 
$\CCyc(G)$ of $G$. 
Replacing the coordinate $x=(x_1, \ldots , x_n)$ of $X=\CC^n$ we may assume that 
$S\in \mathcal{S}_0$ is the origin $\{0\} \subset X=\CC^n$ and for the linear function 
$\ell(x)\coloneq x_1$ we have the condition 
\begin{equation}
d\ell(0)=(0; dx_1) \in \Omega \setminus \bigcup_{S\in \mathcal{S}_0,S\neq \{0\}} \overline{T_S^\ast X}.
\end{equation}
Then for $b\in \CC$ such that $0<\abs{b} \ll 1$ we obtain the transversality 
\begin{equation}
\ell^{-1}(b) \pitchfork S \quad (S\in \mathcal{S}_0,S\neq \{0\})
\end{equation}
on a neighborhood of the origin $0\in X=\CC^n$. 
Using the above coordinate $x=(x_1, \ldots , x_n)=(x_1, x^\prime)$ of $X=\CC^n$ 
we set $L\coloneq \{x^\prime =0\}$ ($=\CC_{x_1}$), 
$H\coloneq \{x_1=0\}$ ($=\CC_{x^\prime}^{n-1}$) $\subset X=\CC^n$ 
so that we have $X=L\times H$. 
For $\e, \delta >0$ let $B(0;\e) \subset L=\CC_{x_1}$ (resp. $B(0;\delta) \subset H=\CC_{x^\prime}^{n-1}$) 
be the open ball centered at $0\in L=\CC_{x_1}$ (resp. $0\in H=\CC_{x^\prime}^{n-1}$) with 
radius $\e>0$ (resp. $\delta >0$). 
Note that by our choice of the coordinate $x=(x_1, x^\prime)$ of $X=\CC^n$ 
for any $S \in \mathcal{S}_0$ such that $S \not= \{ 0 \}$ 
we have $S  \pitchfork H$. We thus obtain a 
Whitney stratification 
\begin{equation}
\mathcal{S}_0^H \coloneq \{S\cap H \mid S\in \mathcal{S}_0\}
\end{equation}
of $Y\cap H \subset H=\CC_{x^\prime}^{n-1}$ induced by $\mathcal{S}_0$. 
Moreover by the microlocal Bertini-Sard theorem \cite[Proposition 8.3.12]{KS90} 
there exists $\delta_0 >0$ such that the submanifold $\partial B(0;\delta) \subset H=\CC_{x^\prime}^{n-1}$ 
intersects strata $S\in \mathcal{S}_0^H$ in $\mathcal{S}_0^H$ transversally 
for any $0<\delta \leq \delta_0$. 
Let $m_0\geq 0$ be the multiplicity of $[T_{\{0\}}^\ast X]$ in the characteristic cycle 
$\CCyc(G)$ of the perverse sheaf $G\simeq \psi_t(\CC_{\{ta=f(x)\}}[n-d])$ 
on $X\simeq X\times \{a\}$. 
Then we obtain the following result. 

\begin{theorem}\label{thm-mult}
In the situation as above, for the holomorphic functions 
$g_i\coloneq f_i|_H \colon H\longrightarrow \CC \quad (1\leq i\leq d)$ 
on $H=\CC_{x^\prime}^{n-1}$ and the morphism $g\coloneq (g_1, 
\ldots , g_d) \colon H \longrightarrow \CC^d$ 
associated to them assume that $Y\cap H =\{g_1=\cdots =g_d=0 \} \subset H$ is a C.I. 
and $a\in \CC^d \setminus C_{\{0\}}(\Delta_g)$. Then we have 
\begin{enumerate}
\item[\rm{(i)}] 
The morphism $g \colon H \longrightarrow \CC^d$ satisfies the Thom condition in 
the direction $a \in \CC^d \setminus C_{\{0\}}(\Delta_g)$
for the Whitney stratification $\mathcal{S}_0^H$ of $Y\cap H =\{g_1=\cdots =g_d =0 \} 
\subset H = \CC_{x^\prime}^{n-1}$. 
\item[\rm{(ii)}] 
For the Milnor fiber $M_{Y\cap H, 0} \subset H\setminus (Y\cap H)$ 
of $g \colon H\longrightarrow \CC^d$ 
at the origin $0\in Y\cap H \subset H$ (in the direction $a$) 
defined by \rm{(i)} we have 
\begin{equation}
m_0=(-1)^{n-d}\left\{ \chi(M_{Y,0})-\chi(M_{Y\cap H, 0})\right\}.
\end{equation}
\end{enumerate}
\end{theorem}

\begin{proof}
Let $U\subset \CC^d$ be a sufficiently small neighborhood 
of the point $a\in \CC^d \setminus 
( \mathrm{C}_{\{0\}}(\Delta_f) \cup \mathrm{C}_{\{0\}}(\Delta_g))$ such that 
the open cone $\CC^\ast U\subset \CC^d$ satisfies the conditions 
for $f \colon X \longrightarrow \CC^d$ in Definition \ref{def-thom}. 
Then by the condition $a \notin \mathrm{C}_{\{0\}}(\Delta_f) 
\cup \mathrm{C}_{\{0\}}(\Delta_g)$ there exists 
$0 < \eta \ll 1$ such that for any point $p \in D^*_{\eta} \cap \CC^\ast U$ 
its fiber $f^{-1}(p) \subset X$ (resp. $g^{-1}(p) \subset H$) 
by $f \colon X \longrightarrow \CC^d$ (resp. $g \colon H \longrightarrow \CC^d$) 
is smooth. Suppose there exist 
a sequence of points $p_i\in D^*_{\eta} \cap \CC^\ast U$ 
($i=1,2,3, \ldots $) and that of 
points $q_i\in g^{-1}(p_i)$ ($i=1,2,3, \ldots $) satisfying the conditions 
\begin{equation}
\begin{cases}
q_i \longrightarrow ^\exists q \in S \cap H \subset Y \cap H \quad 
(i\longrightarrow +\infty), \\
T_{q_i}(g^{-1}(p_i)) \longrightarrow 
^\exists \mathbf{T}_0 \ (\simeq \CC^{n-1-d})\subset T_qH \quad 
(i\longrightarrow +\infty )
\end{cases}
\end{equation}
for some $S \in \mathcal{S}_0$ such that $S \not= \{ 0 \}$. 
By taking their subsequences, we may assume that 
\begin{equation}
T_{q_i}(f^{-1}(p_i)) \longrightarrow 
^\exists \mathbf{T} (\simeq \CC^{n-d})\subset T_qX \quad 
(i\longrightarrow +\infty ). 
\end{equation}
Then by the Thom condition of $f$ in the direction $a$ 
we obtain $\mathbf{T} \supset T_qS$. Note that by the 
smoothness of $g^{-1}(p_i) \subset H$ the complex 
submanifold $f^{-1}(p_i) \subset X$ intersects 
the hyperplane $H \subset X$ transversally. Recall also 
that for the stratum $S \in \mathcal{S}_0$ such that $S \not= \{ 0 \}$ 
we have the transversality $S \pitchfork H$. 
This implies that $\mathbf{T} + T_qH=T_qX$. Then we can 
easily see that $\mathbf{T}_0 = \mathbf{T} \cap T_qH$ and hence 
\begin{equation}
\mathbf{T}_0 \supset T_qS \cap T_qH=T_q(S \cap H). 
\end{equation}
We thus obtain the assertion (i). Let us prove (ii). 
We apply the arguments in Section \ref{sec-pre} to the 
Whitney stratification $\mathcal{S}_0$ of 
$Y\subset X=\CC^n =L\times H$. 
Then there exist $0<\e_0 \ll 1$ and a 1-dimensional subanalytic subset $E\subset \RR^2$ 
contained in 
$\{ (\e, \delta)\in \RR^2 \mid 0<\e \leq \e_0, \quad 0<\delta \leq \delta_0 \} \subset \RR^2$ 
such that for any $0<\e \leq \e_0$ and $0<\delta \leq \delta_0$ satisfying the condition 
$(\e, \delta)\not \in E$ the three submanifolds 
\begin{equation}
\partial B(0;\e)\times \partial B(0;\delta), \quad 
\partial B(0;\e)\times B(0;\delta), \quad 
B(0;\e)\times \partial B(0;\delta)
\end{equation} 
of $X=\CC^n=L\times H$ intersect any stratum $S\in \mathcal{S}_0$ in $\mathcal{S}_0$ transversally. 
As in the proof of Proposition \ref{prop-versp}, we can show that there exist 
$0<\delta_1 \leq \delta_0$ and a subanalytic function $\phi 
\colon (0,\delta_1) \longrightarrow \RR_{>0}$ 
such that 
\begin{equation}
\{ (\e, \delta) \in \RR^2 \mid 0<\delta \leq \delta_1, \ 0<\e< \phi(\delta) \} \cap E =\varnothing.
\end{equation}
Now we apply the proof of \cite[Theorem 2.6]{Tak25} to the nearby cycle preverse sheaf 
$G=\psi_t(\CC_{\{ta=f(x)\}}[n-d])$. 
Then with the help of \cite[Th\'eor\`eme 4.2.1]{BMM94} and the above transversality, 
we can show that for any $(\e, \delta)\in \RR^2$ such that $0<\delta \leq \delta_1, \quad 
0<\e <\phi(\delta)$ there exist isomorphisms 
\begin{equation}
\rsect(B(0;\e)\times B(0;\delta); \CC_{\{ \tau a=f(x)\}}[n-d]) 
\simto G_0 \quad (0<\abs{\tau} \ll 1).
\end{equation}
Recall also that by Proposition \ref{prop-milfib} we have isomorphisms 
\begin{equation}
H^jG_0 \simeq H^{j+n-d}(M_{Y,0}; \CC) \quad (j\in \ZZ).
\end{equation}
Moreover for $0<\delta \ll \delta_1$ there exist natural isomorphisms 
\begin{align}
&H^j (\{0\} \times B(0;\delta); \CC_{\{ \tau a=f(x)\}}[n-d]) \\
&\simeq H^j (B(0;\delta); \CC_{\{ \tau a=g(x)\}}[n-d]) \\
&\simeq H^{j+n-d}(M_{Y\cap H, 0}; \CC) \quad (0<\abs{\tau} \ll 1, j\in \ZZ).
\end{align}
Fix such $0<\delta \ll \delta_1$ and denote it by $\delta_2$. 
Then by the Thom condition of $f$ in the direction $a$, 
there exists $0<\e_1<\phi(\delta_2)$ such that for any $0<\e \leq \e_1$ the two submanifolds 
\begin{equation}
\partial B(0;\e)\times \partial B(0;\delta_2), \quad B(0;\e)\times \partial B(0;\delta_2)
\end{equation}
of $X=\CC^n$ intersect the complex submanifolds 
\begin{equation}
Z_\tau \coloneq \{ x\in X \mid \tau a =f(x) \}\subset X=\CC^n \quad (0< \abs\tau \ll 1)
\end{equation}
transversally. 
Indeed, otherwise, there exists a sequence of points $\tau_i \in \CC^\ast$ 
($i=1,2,3, \ldots $) such that $\tau_i \longrightarrow 0$ 
($i\longrightarrow +\infty$) and that of points 
\begin{equation}
q_i \in Z_{\tau_i} \cap ( B(0;\e_1) \setminus \{ 0 \} )
\times \partial B(0;\delta_2) \quad 
(i=1,2,3, \ldots )
\end{equation}
such that for any $i=1,2,3, \ldots $ and 
the positive real number $0< r_i < \e_1$ 
defined by the condition $q_i \in \partial B(0; r_i)
\times \partial B(0;\delta_2)$ the complex submanifold 
$Z_{\tau_i a}=f^{-1}( \tau_i a) \subset X$ does not 
intersect $\partial B(0; r_i)
\times \partial B(0;\delta_2) \subset X$ transversally at $q_i$. 
By taking their subsequences, we may assume that 
\begin{equation}
\begin{cases}
q_i \longrightarrow ^\exists q \in S \subset Y \quad 
(i\longrightarrow +\infty), \\
T_{q_i}(f^{-1}( \tau_i a)) =T_{q_i}Z_{\tau_i a} 
\longrightarrow 
^\exists \mathbf{T} (\simeq \CC^{n-d})\subset T_qX \quad 
(i\longrightarrow +\infty )
\end{cases}
\end{equation}
for some $S \in \mathcal{S}_0$ such that $S \not= \{ 0 \}$. 
Then by the Thom condition of $f$ in the direction $a$ 
we obtain $\mathbf{T} \supset T_qS$. In the case 
where $q \in  ( \overline{B(0;\e_1)} \setminus \{ 0 \} )
\times \partial B(0;\delta_2)$ we define a 
positive real number $0< r \leq \e_1$ 
by the condition $q \in \partial B(0; r)
\times \partial B(0;\delta_2)$. Then the inclusion 
$\mathbf{T} \supset T_qS$ implies that the stratum 
$S \in \mathcal{S}_0$ 
does not intersect 
$\partial B(0; r)
\times \partial B(0;\delta_2) \subset X$ transversally at $q$. 
This is a contradiction. Also in the case where 
$q \in  \{ 0 \} \times \partial B(0;\delta_2)$, 
by using the transversality $S  \pitchfork 
(\{ 0 \} \times \partial B(0;\delta_2))$ for 
the stratum $S \in \mathcal{S}_0$ such that $S \not= \{ 0 \}$ 
we obtain a contradiction. 
On the other hand, the condition $d\ell(0) =(0;dx_1) \in \Omega \subset T_{\{0\}}^\ast X$ 
implies that for $0< \abs\tau \ll 1$ the restriction 
$\ell|_{Z_\tau} : Z_\tau \longrightarrow \CC$ of the linear function 
$\ell : X=\CC^n \longrightarrow \CC$ to $Z_\tau \subset X=\CC^n$ 
has only non-degenerate (complex Morse) critical points on a neighborhood of the origin 
$0\in X=\CC^n$ and the number of them is equal 
to $m_0\geq 0$ (see the proof of \cite[Theorem 5.5]{FKT26}). 
Note that at each of such critical points the restriction of the real analytic 
function $\abs{\ell}^2 \colon X\longrightarrow \RR$ to $Z_\tau \subset X=\CC^n$ 
has a non-degenerate (real Morse) critical point of Morse index $n-d =\dim Z_\tau$. 
Then for $0< \abs\tau \ll 1$ by applying the non-characteristic deformation lemma 
\cite[Proposition 2.7.2]{KS90} to $\rsect_{L\times B(0;\delta_2)} (\CC_{Z_\tau}) \in \BDC_{\RR-c}(X)$ 
we immediately obtain the assertion (ii). 
This completes the proof.
\end{proof}
Let us consider the remaining strata $S\in \mathcal{S}_0$ such that $\dim S\geq 1$. 
For such a stratum $S\in \mathcal{S}_0$ we take an affine subspace 
$W_S (\simeq \CC^{n-\dim S})$ of $X=\CC^n$ of dimension $n-\dim S$ which intersects 
$S$ at a point $p_S \in S$ transversally. 
We call it a normal slice of $S\in \mathcal{S}_0$ at $p_S \in S$. 
Then there exists a neighborhood $U_S \subset W_S$ of $p_S \in W_S\cap S$ in 
$W_S \simeq \CC^{n-\dim S}$ such that $\{ S^\prime \cap U_S \mid S^\prime \in \mathcal{S}_0\}$ 
is a Whitney stratification of $Y\cap U_S \subset W_S$ and 
$G_S \coloneq G|_{U_S}[-\dim S] \in \BDC_c(U_S)$ is a perverse sheaf 
on $U_S \subset W_S \simeq \CC^{n-\dim S}$. 
Moreover, as in the proof of Proposition \ref{prop-milfib} we can easily show that 
there exists an isomorphism 
\begin{equation}
G_S \simeq \psi_t(\CC_{\{ ta=f(x)\}}|_{U_S\times \CC}[n-d-\dim S]).
\end{equation} 
Let $m_S\geq 0$ be the multiplicity of $[T_{\{p_S\}}^\ast U_S]$ 
in the characteristic cycle $\CCyc(G_S)$ of the perverse sheaf $G_S$ on $U_S$. 
Then as in Section \ref{sec-hypersurface} we can show that 
the multiplicity of $[T_S^\ast X]$ in the 
characteristic cycle $\CCyc(G)$ of the original perverse sheaf $G$ on $X$ 
is equal to $(-1)^{\dim S}m_S \in \ZZ$. 
Moreover, as in Theorem \ref{thm-mult} by taking a generic hyperplane $H_S \subset W_S$ 
in $W_S \simeq \CC^{n-\dim S}$ passing through the point $p_S \in W_S \cap S$ 
and assuming some conditions we can show that 
\begin{equation}
m_S =(-1)^{n-d-\dim S}\left\{ \chi(M_{Y\cap W_S, p_S})-\chi(M_{Y\cap H_S, p_S})\right\}.
\end{equation}
We leave the precise formulation to the readers (see Theorem \ref{thm-coef} below). 
From now, we consider the case where $0\in Y$ and the origin $\{0\}$ of $X=\CC^n$ 
is a stratum of $\mathcal{S}_0$. 
Then there exists a sufficiently small neighborhood $V$ of the origin $0$ in $X=\CC^n$ 
such that $S\cap V$ is connected and $\{0\} \subset \overline{S}$ for any $S\in \mathcal{S}_0$ 
such that $S\cap V\neq \varnothing$. 
We thus obtain a Whitney stratification $Y\cap V=\bigsqcup_{\alpha \in A}Y_{\alpha}$ 
of $Y\cap V$. 
We denote $\alpha \in A$ such that $Y_{\alpha}=\{0\}$ by $\alpha_0$. 
For $\alpha \in A$ such that $\alpha \neq \alpha_0$ and $\dim Y_{\alpha} <\dim(Y\cap V)$ 
let $W_{\alpha} (\simeq \CC^{n-\dim Y_{\alpha}}) \subset X=\CC^n$ be a normal 
slice of $Y_{\alpha}$ at a point $p_{\alpha} \in Y_{\alpha}$ and $U_{\alpha} \subset W_{\alpha}$ 
a sufficiently small neighborhood of $p_\alpha \in W_{\alpha} \cap Y_{\alpha}$ in $W_{\alpha}$. 
Let $m_{\alpha}\geq 0$ be the multiplicity of $[T_{\{p_{\alpha}\}}^\ast U_{\alpha}]$ 
in the characteristic cycle of the perverse sheaf 
$G_{\alpha}\coloneq G|_{U_{\alpha}}[-\dim Y_{\alpha}] 
\in \Perv(\CC_{U_{\alpha}})$ on $U_{\alpha}$ and take a 
generic hyperplane $H_{\alpha} \subset W_{\alpha}$ 
in $W_{\alpha} \simeq \CC^{n-\dim Y_{\alpha}}$ passing through 
the point $p_{\alpha} \in W_{\alpha} \cap Y_{\alpha}$. 
Then we obtain the following result. 

\begin{theorem}\label{thm-coef}
In the situation as above, for the holomorphic functions 
$g_{\alpha, i}\coloneq f_i|_{H_\alpha \cap U_\alpha} : H_\alpha \cap 
U_\alpha \longrightarrow \CC \quad (1\leq i\leq d)$ 
on $H_\alpha \cap U_\alpha \subset W_\alpha$ and the morphism 
$g_\alpha \coloneq (g_{\alpha, 1}, \ldots , g_{\alpha, d}) \colon 
H_\alpha \cap U_\alpha \longrightarrow \CC^d$ 
associated to them, assume that $Y\cap H_\alpha \cap U_\alpha 
=\{g_{\alpha,1}=\cdots =g_{\alpha,d}=0\} \subset H_\alpha \cap U_\alpha$ 
is a C.I. and 
$a\in \CC^d \setminus C_{\{0\}}(\Delta_{g_\alpha})$. Then we have 
\begin{enumerate}
\item[\rm{(i)}] 
The morphism $g_\alpha$ 
satisfies the Thom condition in the 
direction $a\in \CC^d \setminus C_{\{0\}}(\Delta_{g_\alpha})$ 
with respect to the natural Whitney 
stratification of $Y\cap H_\alpha \cap U_\alpha$ induced by $\mathcal{S}_0$. 
\item[\rm{(ii)}] 
For the Milnor fiber $M_{Y\cap H_\alpha \cap U_\alpha, p_\alpha} \subset 
(H_\alpha \cap U_\alpha ) \setminus (Y\cap H_\alpha \cap U_\alpha)$ 
of $g_{\alpha} \colon H_\alpha \cap U_\alpha \longrightarrow \CC^d$ 
at the point $p_{\alpha} \in Y\cap  H_\alpha \cap U_\alpha \subset 
H_\alpha \cap U_\alpha$ (in the direction $a$) defined by \rm{(i)} we have 
\begin{equation}
m_\alpha =(-1)^{\dim Y-\dim Y_\alpha}\left\{ \chi(M_{Y\cap 
U_\alpha, p_\alpha})-\chi(M_{Y\cap H_\alpha \cap U_\alpha, p_\alpha})\right\}.
\end{equation}
\end{enumerate}
\end{theorem}

\begin{proof}
By our assumptions, $Y\cap U_\alpha \subset U_\alpha$ is a C.I. and 
the restriction $f_\alpha \coloneq f|_{U_\alpha} \colon U_\alpha 
\longrightarrow \CC^d$ of 
$f\colon X\longrightarrow \CC^d$ satisfies the Thom condition in 
the direction $a\in \CC^d \setminus C_{\{0\}}(\Delta_f)$ 
with respect to the natural Whitney stratification of $Y\cap 
U_\alpha$ induced by $\mathcal{S}_0$. 
Then we obtain the assertions as in the proof of Theorem \ref{thm-mult}. 
\end{proof}

\begin{theorem}
In the situation as above, assume that $a\in \CC^d \setminus C_{\{0\}}(\Delta_g)$ 
and for any $\alpha \in A$ such that $\alpha \neq \alpha_0$ 
and $\dim Y_\alpha < \dim Y$ the condition in Theorem \ref{thm-coef} is satisfied. 
Then we have an equality 
\begin{align}
\mathrm{Eu}_Y(0)=&\chi(M_{Y\cap H, 0}) \\
&-\sum_{\substack{\alpha \neq \alpha_0, 0\in \overline{Y_\alpha} \\ \dim Y_\alpha <\dim Y}}
\left\{\chi(M_{Y\cap U_\alpha, p_\alpha})-\chi(M_{Y\cap H_\alpha \cap U_\alpha, p_\alpha})
\right\} \cdot \mathrm{Eu}_{\overline{Y_\alpha}}(0).
\end{align}
\end{theorem}
\begin{proof}
By the Riemann-Hilbert correspondence there exists a regular holonomic 
$\SD_X$-module $\SM$ on $X=\CC^n$ such that 
\begin{equation}
\Sol_X(\SM)[n]\simeq G =\psi_t(\CC_{\{ta=f(x)\}}[n-d]).
\end{equation}
Then by $n-d=\dim Y$ and Proposition \ref{prop-milfib} we obtain an equality 
\begin{equation}\label{eq-chisol}
\chi(\Sol_X(\SM))(0)=(-1)^{\dim X-\dim Y}\chi(M_{Y,0}).
\end{equation}
For $\alpha \in A$ let $\mathrm{mult}_{T_{Y_\alpha}^\ast X}(\SM) \geq 0$ be the multiplicity 
of $\SM$ along the Lagrangian submanifold $T_{Y_\alpha}^\ast X \subset T^\ast X$. 
Note that for any $\alpha \in A$ such that $\alpha \neq \alpha_0$ and 
$\dim Y_\alpha <\dim Y$ we have $m_\alpha =\mathrm{mult}_{T_{Y_\alpha}^\ast X}(\SM)$. 
Moreover we have $m_0 =\mathrm{mult}_{T_{\{0\}}^\ast X}(\SM)$. 
In this situation, Kashiwara's index theorem in \cite[Chapter 6]{Kas83a} implies that 
\begin{align}
&\chi(\Sol_X(\SM))(0) = (-1)^{\dim X -\dim Y} \mathrm{Eu}_Y(0) \\
&+\sum_{\alpha \neq \alpha_0, \dim Y_\alpha < \dim Y}(-1)^{\dim X 
-\dim Y_\alpha}m_\alpha \cdot \mathrm{Eu}_{\overline{Y_\alpha}}(0) 
+(-1)^n m_0\cdot \mathbf{1}_{\{0\}}.
\end{align}
Then by Theorems \ref{thm-mult} and \ref{thm-coef} and \eqref{eq-chisol} 
the assertion immediately follows.
\end{proof}

\appendix 
\section{A construction of open dense subsets 
$\Omega\subset T_{\{0\}}^\ast X$}\label{ap-A}
In this appendix, we explain how to construct an open dense subset 
$\Omega\subset T_{\{0\}}^\ast X$ which can be used to prove the 
results in Section \ref{sec-hypersurface}. 
To obtain such a subanalytic one $\Omega$, it suffices to recall the proof of 
\cite[Theorem 4.7]{FKT26}. 
Indeed, by \cite[Lemma 5.3]{FKT26} the complex nearby cycle sheaf 
$G=\psi_t({i_f}_\ast (\CC_X[n]))[-1] \in \Perv(\CC_X)$ on $X$ can be obtained by 
the real nearby cycle functor in \cite{FKT26} as follows. 
Let $\RR\coloneq \{\tau \in \CC \mid \mathrm{Im} \tau =0\} \subset \CC$ be the 
real part of $\CC$ and $u\coloneq  \mathrm{Re} t|_{X\times \RR} \colon X\times \RR 
\longrightarrow \RR$ the restriction of the projection 
$ \mathrm{Re} t\colon X\times \RR \longrightarrow \RR$ to 
$X\times \RR \subset X\times \CC$. 
Then by \cite[Lemma 5.3]{FKT26} there exists an isomorphism 
\begin{equation}
G\simeq \psi_u^\RR \left( {i_f}_\ast (\CC_X[n])|_{X\times \RR}\right) [-1]
\end{equation}
and we can apply the arguments in the proof of \cite[Theorem 4.7]{FKT26} to describe 
the characteristic cycle $\CCyc(G)$ of $G$. 
Note that the conormal bundle $T_{\{0\}}^\ast X$ is a union of some strata 
$\Theta_\alpha$ in the proof of \cite[Theorem 4.7]{FKT26}. 
Then it suffices to define $\Omega \subset T_{\{0\}}^\ast X$ to be 
the union of such strata $\Theta_\alpha \subset T_{\{0\}}^\ast X$ satisfying the condition 
$\dim_\RR \Theta_\alpha =\dim_\RR T_{\{0\}}^\ast X$. 
We can also construct a $\CC$-analytic open dense subset $\Omega 
\subset T_{\{0\}}^\ast X$ which can be used prove the results in Section \ref{sec-hypersurface} 
in the following way. 
We denote the (complex analytic) Whitney stratification 
\begin{equation}
i_f(X)=\left( \bigsqcup_{\alpha \in A}i_f(Y_\alpha)\right) \sqcup i_f(X\setminus Y)
\end{equation}
of $\supp ({i_f}_\ast (\CC_X[n]))=i_f(X) \subset X\times \CC$ by $\mathcal{S}$ and 
set $\mathcal{S}_0\coloneq \{S\in \mathcal{S} \mid S\subset X\times \{0\}\} 
=\{Y_\alpha \times \{0\} \}_{\alpha \in A}$. 
Let $D_\e \coloneq \{\tau \in \CC \mid \abs \tau <\e \} \subset \CC \ (0<\e \ll 1)$ 
be a sufficiently small disc in $\CC$ centered at the origin $0\in \CC$ and 
$T^\ast (X\times D_\e / D_\e) \simeq (T^\ast X) \times D_\e$ the relative cotangent 
bundle associated to the projection $X\times D_\e \longrightarrow D_\e$. 
For $a \in D_\e$ we set $X_a \coloneq t^{-1}(a)= 
X\times \{a\} \subset X\times \CC$. 
Then, shrinking $X$ if necessary, we may assume that for any $a\in D_\e^\ast \coloneq 
D_\e \setminus \{0\}$ the smooth complex hypersurface $X_a \subset X\times \CC$ 
intersects any stratum $S\in \mathcal{S} \setminus \mathcal{S}_0$ 
in $\mathcal{S} \setminus \mathcal{S}_0$ transversally. 
In fact, $i_f(X\setminus Y) \in \mathcal{S}\setminus \mathcal{S}_0$ is 
the unique stratum in $\mathcal{S}\setminus \mathcal{S}_0$. 
We define an $\RR$-constructible sheaf $K\in \BDC_{\RR-c}(X\times (-\e,\e))$ 
on $X\times (-\e,\e) =(X\times \RR) \cap (X\times D_\e)$ by 
\begin{equation}
K\coloneq \rsect_{X\times (0,\e)}\left( 
{i_f}_\ast (\CC_X[n])|_{X\times (-\e,\e)} \right)[-1] 
\end{equation}
so that we have $K|_{X_0}\simeq G$. 
Then the micro-support $\msupp(K)$ of $K$ is contained in the set 
\begin{equation}
\bigsqcup_{S \in \mathcal{S}\setminus \mathcal{S}_0} 
T^*_{S \cap (X\times (-\e,\e))}(X\times (-\e,\e) ) =
T^*_{i_f(X\setminus Y) \cap (X\times (-\e,\e))}(X\times (-\e,\e) ) 
\quad \subset T^*(X\times (-\e,\e)). 
\end{equation}
over the open subset $X\times (0,\e)
\subset X\times (-\e,\e)$. 
Indeed, there exists an isomorphism 
\begin{equation}
K|_{X\times (0,\e)}\simeq {i_f}_\ast (\CC_X[n])|_{X\times (0,\e)}[-1]
\end{equation}
and the unique stratum $i_f(X\setminus Y) \in \mathcal{S}\setminus \mathcal{S}_0$ 
in $\mathcal{S}\setminus \mathcal{S}_0$ intersects the real hypersuface 
$X\times (-\e,\e) \subset X\times \CC$ transversally. 
For $a\in D_\e$ we set 
\begin{equation}
\Xi_a \coloneq 
\begin{cases}
\bigsqcup_{S\in \mathcal{S}}T_{S\cap X_a} ^\ast (X_a) \quad (a\neq 0), \\
\\
\bigsqcup_{S\in \mathcal{S}_0}T_S^\ast (X_0) \quad (a=0).
\end{cases}
\end{equation}
We define a subset $\Xi \subset T^\ast (X\times D_\e /D_\e)$ by 
$\Xi \coloneq \bigsqcup_{a\in D_\e} \Xi_a \times \{a\}$, 
where we used the natural identifications 
$T^\ast X_a \simeq T^\ast X \ (a\in D_\e)$. 
Then as in the proof of \cite[Theorem 4.7]{FKT26} we see that $\Xi$ is a closed 
subset of $T^\ast (X\times D_\e/D_\e)$. 
Moreover, as in \cite[Appendix B]{FKT26} we can show that 
$\Xi$ is an analytic subset of $T^\ast (X\times D_\e/D_\e)$. 
Indeed, it suffices to show that it is (complex) constructible. 
Then replacing the sphere bundles in \cite[Appendix B]{FKT26} by 
complex projective bundles, we can apply Lemma \ref{lem-const} to prove 
the (complex) constructibility of $\Xi \subset T^\ast (X\times D_\e/D_\e)$. 
Let $\mathcal{Z} \subset 
T^\ast (X\times D_\e/D_\e)$ be the zero section of $T^\ast (X\times D_\e/D_\e)$ 
and set $\Xi^* \coloneq \Xi \setminus 
\mathcal{Z}$. Let $\mathcal{L}$ be a (complex analytic) Whitney stratification of 
the $\CC$-analytic subset $\Xi^* \subset T^\ast (X\times D_\e/D_\e)$ such 
that for any $S\in \mathcal{S}_0$ the conormal bundle $T_S^\ast X_0 
\setminus T^*_{X_0}X_0 \subset \Xi^*$ with its zero section removed 
is a union of some strata in it. Since $\Xi^*$ is $\CC^*$-conic, 
we may assume that any stratum in $\mathcal{L}$ is $\CC^*$-conic. 
Let $T^\ast(X\times (-\e,\e)/(-\e,\e))$ be the relative cotangent bundle 
associated to the 
projection $X\times (-\e,\e) \longrightarrow (-\e,\e)$. 
Then there exists an isomorphism 
\begin{equation}
(X\times (-\e,\e)) \times_{(X\times D_\e)} T^\ast (X\times D_\e /D_\e) \simto 
T^\ast (X\times (-\e,\e)/(-\e,\e))
\end{equation}
induced by the inclusion map 
$X\times (-\e,\e) =(X\times \RR) \cap (X\times D_\e) \longhookrightarrow X\times D_\e$. 
We thus can consider $T^\ast (X\times (-\e,\e)/(-\e,\e))$ as a submanifold of 
$T^\ast (X\times D_\e /D_\e)$. 
Shrinking $D_\e\subset \CC$ if necessary, we may assume that the real hypersurface 
$T^\ast (X\times (-\e,\e)/(-\e,\e)) \subset T^\ast (X\times D_\e/D_\e)$ 
intersects any stratum $S\in \mathcal{L}$ in $\mathcal{L}$ such that 
$S\subset T^\ast (X\times D_\e^\ast/ D_\e^\ast) \simeq (T^\ast X) \times D_\e^\ast$ 
transversally. 
Then we can easily show that 
\begin{equation}
\mathcal{L}^\prime \coloneq \{ S\cap T^\ast (X\times (-\e,\e)/(-\e,\e)) \ \mid \ S\in \mathcal{L} \}
\end{equation}
is a (subanalytic) stratification of $\Xi^* \cap T^\ast(X\times (-\e,\e)/(-\e,\e))$. 
Note that for any stratum $S\in \mathcal{L}$ such that $S\subset 
(T^\ast X) \times \{0\}$ we have $S\cap T^\ast(X\times (-\e,\e)/(-\e, \e)) =S$ and hence 
$S$ is also a (complex analytic) stratum in $\mathcal{L}^\prime$. 
Now we define $\Omega \subset T_{\{0\}}^\ast X=T_{\{0\}}X_0$ to be the (unique) 
such (complex analytic) stratum $S\in \mathcal{L}^\prime$ in $\mathcal{L}^\prime$ 
satisfying the conditions $S \subset T_{\{0\}}^\ast X \setminus T^*_XX$ 
and $\dim_\CC S=\dim_\CC T_{\{0\}}^\ast X$. 
Then, instead of the (subanalytic) Whitney stratification $\mathcal{L}$ in the 
proof of \cite[Theorem 4.7.]{FKT26}, 
we can use the restriction of the above one $\mathcal{L}^\prime$ to the closed subset 
$\Xi^* \cap ((T^\ast X)\times [0,\e))\subset \Xi^* \cap T^\ast (X\times 
(-\e,\e)/(-\e, \e))$ to prove \cite[Theorem 4.7]{FKT26} for the characteristic cycle 
$\CCyc(G)$ of $G\simeq 
\psi_u^\RR \left({i_f}_\ast (\CC_X[n])|_{X\times \RR} \right)[-1] \in \Perv(\CC_X)$. 
The crucial point is that by Proposition \ref{prop-adapt} the perverse 
sheaf $G=\psi_t\left({i_f}_\ast (\CC_X)[n]\right)[-1] \simeq K|_{X_0}$ is adapted to $\mathcal{S}_0$. 
Using also the above property of $K$, we can perform the arguments in the proof of 
\cite[Theorem 4.7]{FKT26} at any point of the $\CC$-analytic open dense subset 
$\Omega \subset T_{\{0\}}^\ast X$. It is clear that this construction of 
$\Omega$ is applicable also to any perverse sheaf on $X \times \CC$. 
By such $\Omega$ we thus can clarify the proof of \cite[Thoerem 5.5]{FKT26}.


\begin{bibdiv}
\begin{biblist}





\bib{BGB22}{article}{
   author = {Brasselet, J.-P.},
   author={Grulha, Jr., N. G.},
   author={B{\'i}ch, T. N. T.},
   title = {Local Euler obstruction, old and new, III},
   journal = {J. Singul.},
   volume = {25},
   year = {2022},
   pages = {90--122},
}

\bib{BLS00}{article}{
   author = {Brasselet, J.-P.},
   author={L\^e{}, D\~ung Tr\'ang},
   author={Seade, J.},
   title = {Euler obstruction and indices of vector fields},
   journal = {Topology},
   volume = {39},
   year = {2000},
   number = {6},
   pages = {1193--1208},
}

\bib{BSS09}{book}{
   author = {Brasselet, Jean-Paul},
   author={Seade, Jos\'e},
   author={Suwa, Tatsuo},
   title= {Vector fields on singular varieties},
   series = {Lecture Notes in Mathematics},
   volume = {1987},
   publisher = {Springer-Verlag, Berlin},
   year = {2009},
   pages = {xx+225},
}


\bib{BMM94}{article}{
   author={Brian\c con, Jo\"el},
   author={Maisonobe, Philippe},
   author={Merle, Michel},
   title={Localisation de syst\`emes diff\'erentiels, stratifications de
   Whitney et condition de Thom},
   journal={Invent. Math.},
   volume={117},
   date={1994},
   number={3},
   pages={531--550},
}






\bib{Bud13}{misc}{
   author={Budur, Nero},
   title={Bernstein--Sato polynomials and generalizations},
   date={2013},
   note={Lecture notes from the Summer School ``Algebra, Algorithms,
   and Algebraic Analysis,'' Rolduc Abbey, Netherlands},
   url={https://drive.google.com/file/d/1xlSTFfw6Tftp2yYMYYxl6lfcmZpDCqx2/view},
}





\bib{CMSS16}{article}{
   author = {Callejas-Bedregal, R.},
   author={Morgado, M. F. Z},
   author={Saia, M.},
   author={Seade, J.},
   label={CMSS16},
   title = {The {L}\^e-{G}reuel formula for functions on analytic spaces},
   journal = {Tohoku Math. J. (2)},
   volume= {68},
   year = {2016},
   number = {3},
   pages= {439--456},
}








\bib{Del73}{collection}{
   author={Deligne, Pierre},
   title={Le formalisme des cycles \'evanescents, in SGA7 XIII and XIV},
   series={Lecture Notes in Mathematics},
   volume={340},
   publisher={Springer-Verlag, Berlin-New York},
   date={1973},
   pages={82--115, 116--164},
}


\bib{Dim04}{book}{
   author={Dimca, Alexandru},
   title={Sheaves in topology},
   series={Universitext},
   publisher={Springer-Verlag, Berlin},
   date={2004},
   pages={xvi+236},
}

\bib{Dub78}{article}{
   author = {Dubson, A. S.},
   title = {Classes caract\'eristiques des vari\'et\'es singuli\`eres},
   journal = {C. R. Acad. Sci. Paris S\'er. A-B},
   volume = {287},
   year = {1978},
   number = {4},
   pages = {A237--A240},
}



\bib{Ern94}{article}{
   author = {Ernstr\"om, Lars},
   title = {Topological {R}adon transforms and the local {E}uler
              obstruction},
   journal = {Duke Math. J.},
   volume = {76},
   year = {1994},
   number = {1},
   pages = {1--21},
}


\bib{FKT26}{arXiv}{
    author={Fernandes, Ren},
    author={Kudomi, Kazuki},
    author={Takeuchi, Kiyoshi},
    title={Characteristic cycles of real and complex constructible sheaves, revisited},
    year={2026},
    eprint={2603.14821.}
}



\bib{Ful98}{book}{
    author = {Fulton, William},
    title = {Intersection theory},
    edition = {Second},
    publisher = {Springer-Verlag, Berlin},
    year = {1998},
    pages = {xiv+470},
}



\bib{Gaf08}{article}{
      author = {Gaffney, Terence},
      title = {Non-isolated complete intersection singularities and the
              {$A_f$} condition},
      booktitle = {Singularities {I}},
      series = {Contemp. Math.},
      volume = {474},
      pages = {85--93},
      publisher = {Amer. Math. Soc., Providence, RI},
      year = {2008},

}


\bib{GGR19}{article}{
    author={Gaffney, Terence},
    author={Grulha, Nivaldo G., Jr.},
    author={Ruas, Maria A. S.},
    title={The local Euler obstruction and topology of the stabilization of
    associated determinantal varieties},
    journal={Math. Z.},
    volume={291},
    date={2019},
    number={3-4},
    pages={905--930},
 }


\bib{Gin86}{article}{
   author={Ginsburg, V.},
   title={Characteristic varieties and vanishing cycles},
   journal={Invent. Math.},
   volume={84},
   date={1986},
   number={2},
   pages={327--402},
}


\bib{GS81}{article}{
    author={Gonz\'alez-Sprinberg, Gerardo},
    title={L'obstruction locale d'Euler et le th\'eor\`eme de MacPherson},
    book={
       series={Ast\'erisque},
       volume={82-83},
       publisher={Soc. Math. France, Paris},
    },
    date={1981},
    pages={7--32},
}

\bib{GM88}{book}{
   author={Goresky, Mark},
   author={MacPherson, Robert},
   title={Stratified Morse theory},
   series={Ergebnisse der Mathematik und ihrer Grenzgebiete (3)},
   volume={14},
   publisher={Springer-Verlag, Berlin},
   date={1988},
   pages={xiv+272},
}

\bib{Ham71}{article}{
    author = {Hamm, Helmut},
    title = {Lokale topologische {E}igenschaften komplexer {R}\"aume},
    journal = {Math. Ann.},
    volume = {191},
    year = {1971},
    pages = {235--252},
}

\bib{HL73}{article}{
   author = {Hamm, Helmut A.},
   author={L\^e{}, D\~ung Tr\'ang},
   title = {Un th\'eor\`eme de {Z}ariski du type de {L}efschetz},
   journal = {Ann. Sci. \'Ecole Norm. Sup. (4)},
   volume= {6},
   year = {1973},
   pages = {317--355},
}

\bib{HIO88}{book}{
    author = {Herrmann, M.},
    author={Ikeda, S},
    author={Orbanz, U.},
    title = {Equimultiplicity and blowing up},
    note = {An algebraic study,
              With an appendix by B.\ Moonen},
    publisher = {Springer-Verlag, Berlin},
    year = {1988},
    pages = {xviii+629},
}

\bib{Hir76}{article}{
    author = {Hironaka, Heisuke},
    title = {Stratification and flatness},
    booktitle = {Real and complex singularities ({P}roc. {N}inth {N}ordic
              {S}ummer {S}chool/{NAVF} {S}ympos. {M}ath., {O}slo, 1976)},
    pages = {199--265},
    publisher = {Sijthoff \& Noordhoff, Alphen aan den Rijn},
    year = {1977},
}




\bib{HTT08}{book}{
   author={Hotta, Ryoshi},
   author={Takeuchi, Kiyoshi},
   author={Tanisaki, Toshiyuki},
   title={$D$-modules, perverse sheaves, and representation theory},
   series={Progress in Mathematics},
   volume={236},
   publisher={Birkh\"auser Boston, Inc., Boston, MA},
   date={2008},
   pages={xii+407},
}








 \bib{Kas73}{article}{
    author={Kashiwara, Masaki},
    title={Index theorem for a maximally overdetermined system of linear
    differential equations},
    journal={Proc. Japan Acad.},
    volume={49},
    date={1973},
    pages={803--804},
 }


\bib{Kas83a}{book}{
   author={Kashiwara, Masaki},
   title={Systems of microdifferential equations},
   series={Progress in Mathematics},
   volume={34},
   publisher={Birkh\"auser Boston, Inc., Boston, MA},
   date={1983},
   pages={xv+159},
}


\bib{Kas83b}{article}{
   author={Kashiwara, Masaki},
   title={Vanishing cycle sheaves and holonomic systems of differential
   equations},
   conference={
      title={Algebraic geometry},
      address={Tokyo/Kyoto},
      date={1982},
   },
   book={
      series={Lecture Notes in Math.},
      volume={1016},
      publisher={Springer, Berlin},
   },
   date={1983},
   pages={134--142},
}




\bib{Kas85}{article}{
   author={Kashiwara, Masaki},
   title={Index theorem for constructible sheaves},
   journal={Ast\'erisque},
   number={130},
   date={1985},
   pages={193--209},
}


\bib{Kas03}{book}{
   author={Kashiwara, Masaki},
   title={$D$-modules and microlocal calculus},
   series={Translations of Mathematical Monographs},
   volume={217},
   publisher={American Mathematical Society, Providence, RI},
   date={2003},
   pages={xvi+254},
}






\bib{KS85}{article}{
   author={Kashiwara, Masaki},
   author={Schapira, Pierre},
   title={Microlocal study of sheaves},
   journal={Ast\'erisque},
   number={128},
   date={1985},
   pages={235},
}


\bib{KS90}{book}{
   author={Kashiwara, Masaki},
   author={Schapira, Pierre},
   title={Sheaves on manifolds},
   series={Grundlehren der mathematischen Wissenschaften},
   volume={292},
   publisher={Springer-Verlag, Berlin},
   date={1990},
   pages={x+512},
}















 


  

\bib{Kat78}{article}{
    author={Kato, M.},
    title={Singularities and some global topological properties},
    journal={Proceedings of R.I.M.S. Singularities Symposium},
    year={1978},
}


\bib{Le73}{article}{
   author={L\^e{}, D\~ung Tr\'ang},
   title={Calcul du nombre de cycles \'evanouissants d'une hypersurface
   complexe},
   journal={Ann. Inst. Fourier (Grenoble)},
   volume={23},
   date={1973},
   number={4},
   pages={261--270},
}

\bib{Le75}{article}{
   author = {L\^e{}, D\~ung Tr\'ang},
   title = {La monodromie n'a pas de points fixes},
   journal = {J. Fac. Sci. Univ. Tokyo Sect. IA Math.},
   volume = {22},
   year= {1975},
   number = {3},
   pages = {409--427},
}

\bib{LT81}{article}{
    author = {L\^e{} D. T.},
    author={Teissier, B.},
    label={LT81},
    title = {Vari\'et\'es polaires locales et classes de {C}hern des
              vari\'et\'es singuli\`eres},
    journal = {Ann. of Math. (2)},
    volume = {114},
    year = {1981},
    number = {3},
    pages = {457--491},
}

\bib{Loo84}{book}{
   author = {Looijenga, E. J. N.},
   title= {Isolated singular points on complete intersections},
   series = {London Mathematical Society Lecture Note Series},
   volume= {77},
   publisher = {Cambridge University Press, Cambridge},
   year= {1984},
   pages = {xi+200},
}

\bib{LR22}{article}{
   author = {L\H orincz, A. C.},
   author={Raicu, C.},
   title = {Local {E}uler obstructions for determinantal varieties},
   journal = {Topology Appl.},
   volume = {313},
   year = {2022},
   pages = {Paper No. 107984, 21},
}


\bib{Mac74}{article}{
    author={MacPherson, R. D.},
    title={Chern classes for singular algebraic varieties},
    journal={Ann. of Math. (2)},
    volume={100},
    date={1974},
    pages={423--432},
}







\bib{MT08}{article}{
   author = {Matsui, Yutaka},
   author={Takeuchi, Kiyoshi},
   title = {Topological Radon transforms and degree formulas for dual
              varieties},
   journal = {Proc. Amer. Math. Soc.},
   volume = {136},
   year = {2008},
   number = {7},
   pages = {2365--2373},
}



\bib{MT11}{article}{
    author={Matsui, Yutaka},
    author={Takeuchi, Kiyoshi},
    title={A geometric degree formula for $A$-discriminants and Euler
    obstructions of toric varieties},
    journal={Adv. Math.},
    volume={226},
    date={2011},
    number={2},
    pages={2040--2064},
}


 \bib{Mil68}{book}{
    author={Milnor, John},
    title={Singular points of complex hypersurfaces},
    series={Annals of Mathematics Studies},
    volume={No. 61},
    publisher={Princeton University Press},
    date={1968},
    pages={iii+122},
 }





\bib{Pie88}{book}{
   author = {Piene, R.},
   title = {Cycles polaires et classes de {C}hern pour les vari\'et\'es
              projectives singuli\`eres},
   BOOKTITLE = {Introduction \`a{} la th\'eorie des singularit\'es, {II}},
   series = {Travaux en Cours},
   volume = {37},
   pages = {7--34},
   publisher = {Hermann, Paris},
   year = {1988},
}

\bib{SKK}{article}{
    author = {Sato, Mikio},
    author={Kawai, Takahiro},
    author={Kashiwara, Masaki},
    label={SKK},
    title = {The theory of pseudodifferential equations in the theory of
              hyperfunctions},
    journal = {S\=ugaku},
    volume = {25},
    year = {1973},
    pages = {213--238},
}















\bib{Sch03}{book}{
   author={Sch\"urmann, J\"org},
   title={Topology of singular spaces and constructible sheaves},
   series={Mathematical Monographs (New Series)},
   volume={63},
   publisher={Birkh\"auser Verlag, Basel},
   date={2003},
   pages={x+452},
}



\bib{Sea06}{book}{
   author = {Seade, Jos\'e},
   title = {On the topology of isolated singularities in analytic spaces},
   series = {Progress in Mathematics},
   volume = {241},
   publisher = {Birkh\"auser Verlag, Basel},
   year = {2006},
   pages = {xiv+238},
}

\bib{Sea19}{article}{
   author = {Seade, Jos\'e},
   title = {On Milnor's fibration theorem and its offspring after 50 years},
   journal = {Bull. Amer. Math. Soc. (N.S.)},
   volume = {56},
   year = {2019},
   number = {2},
   pages = {281--348},
}




\bib{Tak25}{article}{
   author={Takeuchi, Kiyoshi},
   title={Geometric monodromies, mixed Hodge numbers of motivic Milnor
   fibers and Newton polyhedra},
   conference={
      title={Handbook of geometry and topology of singularities VII},
   },
   book={
      publisher={Springer},
   },
   date={2025},
   pages={643--720},
}

\bib{Tro20}{article}{
    author = {Trotman, David},
    title = {Stratification theory},
    booktitle = {Handbook of geometry and topology of singularities. {I}},
    pages = {243--273},
    publisher = {Springer, Cham},
    year = {2020},
}

\bib{VPV10}{article}{
   author = {Van Proeyen, Lise},
   author={Veys, Willem},
   title = {The monodromy conjecture for zeta functions associated to
              ideals in dimension two},
   journal = {Ann. Inst. Fourier (Grenoble)},
   volume = {60},
   year = {2010},
   number = {4},
   pages = {1347--1362},
}

\bib{Ver83}{article}{
   author={Verdier, J.-L.},
   title={Sp\'ecialisation de faisceaux et monodromie mod\'er\'ee},
   book={
      series={Ast\'erisque},
      volume={101-102},
      publisher={Soc. Math. France, Paris},
   },
   date={1983},
   pages={332--364},
}
\end{biblist}
\end{bibdiv}
\end{document}